\documentclass[12pt,reqno]{amsart}

\usepackage{amsxtra}
\usepackage{amssymb}
\usepackage{lscape}
\usepackage{color}
\usepackage{multicol}
\usepackage{setspace}
\usepackage{multirow}
\usepackage[a4paper, margin=1.9cm]{geometry}

\usepackage{graphicx}

\usepackage{enumitem}
\newcommand{\printvalues}{topsep=\the\topsep; itemsep=\the\itemsep; parsep=\the\parsep; partopsep=\the\partopsep}

\newcommand{\ip}{\langle\cdot,\cdot\rangle}%\usepackage{multirow} 
\newcommand{\alt}{\raise1pt\hbox{$\bigwedge$}}

\newcommand\x{\xi}

\newcommand\la{\langle}
\newcommand\ra{\rangle}

\newcommand\zz{{\mathfrak z}}

\newcommand\so{{\mathfrak {so}}}

\newcommand\ggo{{\mathfrak g}}
\newcommand\aff{\mathfrak {aff}}

\newcommand{\on}{\operatorname}

\newcommand\ee{{\mathbf e}}

\newcommand\RR{\mathbb R}

\newcommand\ad{\operatorname{ad}}

\newcommand\id{\operatorname{Id}}

\newcommand\hol{\operatorname{Hol}}
\newcommand\Hol{\operatorname{Hol}}

\theoremstyle{plain}
\newcommand{\ri}{{\rm (i)}}
\newcommand{\rii}{{\rm (ii)}}
\newcommand{\riii}{{\rm (iii)}}
\newcommand{\riv}{{\rm (iv)}}

\newtheorem{thm}{Theorem}[section]
\newtheorem{lem}[thm]{Lemma}
\newtheorem{prop}[thm]{Proposition}
\newtheorem{cor}[thm]{Corollary}

\theoremstyle{definition}

\theoremstyle{remark}
\newtheorem{rem}{Remark}[section]

\newtheorem{ex}{Example}

\title{ Invariant Conformal Killing-Yano 2-forms on five dimensional Lie Groups}

\author{A. Herrera}
\email{cecilia.herrera@unc.edu.ar}
\address{FCEFyN, Universidad Nacional de C\'ordoba, 5000 C\'ordoba, Argentina}

\author{M. Origlia}
\email{marcos.origlia@unc.edu.ar}
\address{FaMAF-CIEM, Universidad Nacional de C\'{o}rdoba, 5000 C\'{o}rdoba, Argentina}

\thanks{
The authors were partially supported by
CONICET, ANPCyT and SECyT-UNC (Argentina). The second author was also supported by the Research Foundation - Flanders (FWO Project G.0F93.17N) and ARC (Australia) DP190100317.}

\dedicatory{}

\begin{document}

\begin{abstract}
We study left invariant conformal Killing-Yano (CKY) $2$-forms on Lie groups endowed with a left invariant metric. 
We classify all $5$-dimensional metric Lie algebras carrying CKY tensors that are obtained as a one-dimensional central extension of  4-dimensional metric Lie algebras endowed with a invertible parallel skew-symmetric tensor.
On the other hand, we also classify $5$-dimensional metric Lie algebras with center of dimension greater than one admitting strict CKY tensors. In addition, we determine all possible CKY tensors on these metric Lie algebras. In particular, we exhibit the  first examples of  CKY $2$-forms on metric Lie algebras which do not admit any Sasakian structure. %Finally, for each metric Lie algebra we analyze the holonomy group of the associated simply connected Lie group with the induced left invariant metric.

\end{abstract}

\maketitle
\section{Introduction}
A differential $p$-form $\eta$ on a $n$-dimensional Riemannian manifold $(M,g)$ is called \textit{conformal Killing-Yano} (CKY for short) if it satisfies for any vector field $X$ the following equation
\begin{equation}\label{ckyManifolds}
\nabla_X  \eta=\dfrac{1}{p+1}\iota_X\mathrm{d}\eta-\dfrac{1}{n-p+1}X^*\wedge \mathrm{d}^*\eta,
\end{equation}
where $X^*$ is the dual 1-form of $X$,  $\mathrm{d}^*$ is the co-differential operator, $\nabla$ is the Levi-Civita connection associated to $g$ and $\iota_X$ is the interior product with $X$. If $\eta$ is co-closed ($\mathrm d^*\eta=0$) then $\eta$ is said to be a Killing-Yano  $p$-form (KY for short). % Notice that a parallel form is automatically Killing Yano form.
Closed conformal Killing-Yano forms will be called $*$-Killing-Yano $p$-forms. Note that the Hodge-star operator 
applied to a CKY $p$-form determines a CKY $(n-p)$-form on $M$. In particular, the Hodge-star operator interchanges closed and co-closed CKY forms on $M$.

Killing-Yano $p$-forms were defined by  K. Yano  in  1951 \cite{Yano1}, as a natural generalization of  Killing vector fields. Indeed, for $p=1$ they are dual to Killing vector fields.
Later, Kashiwada and Tachibana studied CKY forms in \cite{Tachibana}, \cite{Kashiwada} as a generalization of KY forms. Indeed, it is easy to see that a CKY $1$-form is dual to a conformal vector field.

This topic has taken a lot of interest not only in geometry but also in physics, since it has applications in the study of black holes.  A CKY form describes hidden symmetries, that is symmetries that cannot be identified directly.  In 1970, Walker and Penrose noticed in \cite{Penrose-Walker} that some  conserved quantity in the Kerr metric is originated from a Killing tensor of rank two. Since then, the study of CKY forms is in continuous growth in black holes theory.

In the literature, there are more results concerning the (non-)existence of KY forms than ``strict'' conformal Killing-Yano forms, that is, CKY forms which are not KY. Notable examples of KY $2$-forms are given by nearly K\"ahler manifolds $(M,g,J)$ whose fundamental $2$-form $\omega$ given by $\omega(X,Y)=g(JX,Y)$ is KY. For a compact simply connected symmetric space $M$,  it was shown in  \cite{BMS} that $M$ admits a non-parallel Killing-Yano $p$-form, $p \geq 2$, if and only if it is isometric to a Riemannian product $ S^ k \times N$, where $S^ k$ is a round sphere and $k > p$. Another important result states that every Killing-Yano $p$-form on a compact quaternion K\"ahler manifold is parallel for any $p \geq 2 $, see \cite{Moroianu-semmelmann}.

In the case of CKY forms, the following results were proved  in \cite{Semmelmann}: the canonical $2$-form of a Sasakian manifold is a strict CKY $2$-form; a CKY $p$-form ($p  \neq  3, 4$) on a compact manifold with holonomy $G_ 2$  is parallel; the vector space of CKY $p$-forms on a $n$-dimensional connected Riemannian manifold has dimension at most ${n+2}\choose{p+1}$; and there are no CKY forms on compact
manifolds of negative constant sectional curvature. 
In \cite{Moroianu-semmelmann-08} a description of CKY $p$-forms on a compact Riemannian product was given, proving that such a form is a sum of forms of the following types: parallel forms, pull-back of Killing-Yano forms on the factors, and their Hodge duals.

%In this article we study left invariant CKY $2$-forms on Lie groups endowed with a left invariant metric. The study of KY 2-forms in this setting begun in \cite{BDS} and was continued in \cite{AD20} and \cite{dBM19}.  Left invariant CKY 2-forms on $3$-dimensional Lie groups were considered in \cite{ABD}. 
%CKY $2$-forms (not necessary left invariant) on four dimensional simply connected Lie groups with left invariant metrics were classified in \cite{ABM},  where it is also shown that those $2$-forms which are left invariant are  parallel. 
%Some general properties of strict CKY $2$-forms were described in \cite{AD}, where the authors also give a characterization of strict CKY $2$-forms. 
%Recently, in \cite{dBM21} CKY $2$- and $3$-forms were study in $2$-step nilpotent Lie groups.
%Besides this progress, we notice that there are not many examples of strict CKY $2$-forms on Lie groups, except for the ones induced by left invariant Sasakian structures. 
%{\color{blue}The first goal of this work is to fill this gap by giving examples of non-Sasakian Lie algebras that admit a CKY $2$-form. Secondly, we provide a classification of $5$-dimensional metric Lie algebras admitting strict CKY $2$-forms distinguishing two cases according the dimension of their center. }

%%%%%%

In this article we study left invariant CKY $2$-forms on Lie groups endowed with a left invariant metric. The study of KY 2-forms in this setting began in \cite{BDS} and was continued in \cite{AD20} and \cite{dBM19}. 
CKY $2$-forms (not necessary left invariant) on four dimensional simply connected Lie groups with left invariant metrics were classified in \cite{ABM},  where it is also shown that those $2$-forms which are left invariant are  parallel. Recently, in \cite{dBM21} CKY $2$- and $3$-forms were studied in $2$-step nilpotent Lie groups.

Some general properties of strict CKY $2$-forms were described in \cite{AD}. In particular \cite[Theorem $4.3$]{AD} shows  that a metric Lie algebra admitting a strict CKY tensor has odd dimension.  In dimension $3$, all CKY $2$-forms were classified in \cite{ABD}. It is not clear yet what happens in dimensions greater than three in general. It is known that the only $2$-step nilpotent Lie algebras with a strict CKY 2-form are the Heisenberg Lie algebras of dimension $2n+1$ \cite[Theorem 4.4]{ABD} (see also \cite{dBM21}), which also admit Sasakian structures.
%, moreover it is the only $(2n + 1)$-dimensional nilpotent Lie algebra admitting Sasakian structures according to \cite[Theorem $5$]{AFV}.

Besides this progress, it is clear that there are not many explicit examples of strict CKY $2$-forms on Lie groups, except for the ones induced by left invariant Sasakian structures; and the next natural case is to consider $5$-dimensional metric Lie algebras.
Therefore, the first goal of this work is to fill this gap by giving examples of non-Sasakian Lie algebras that admit a CKY $2$-form. Secondly, we provide a classification of $5$-dimensional metric Lie algebras admitting strict CKY $2$-forms distinguishing between two cases according to the dimension of their center. 

%%%%%%%%%%%%

This work is organized as follows: In Section 
$2$ we recall some basic definitions about left invariant CKY 2-forms on Lie groups endowed with a left invariant metric. We also review some results from \cite{AD} that will be useful in this work.
In Section $3$ we use the recent classification of $4$-dimensional metric Lie algebra carrying parallel skew-symmetric endomorphisms in \cite{Herrera}, to classify all $5$-dimensional metric Lie algebras with one dimensional center admitting a CKY $2$-form whose co-differential lies in the dual of the center (Theorem \ref{classificacion_1}).

In Section $4$ we exhibit some general results on metric Lie algebras with center of dimension greater than one carrying a CKY $2$-form. 
We continue in Section $5$ with a complete classification, up to isometric isomorphism and scaling, of $5$-dimensional metric Lie algebras with center of dimension greater than one admitting a CKY $2$-form (Theorem \ref{metricLA}). In particular, we obtain the first explicit examples of CKY tensors on metric Lie algebras which do not admit any Sasakian structure.

In Section $6$ we study the vector space of left invariant CKY $2$-forms on each family of metric Lie algebras obtained in Section $3$ and Section $5$. 
In the former case we show that the CKY $2$-forms are always closed (Proposition \ref{closed}), in contrast with the latter case when they are never closed (Proposition \ref{never_closed}). Moreover, in Theorem \ref{CKYspace_extensiones} and Theorem \ref{CKYspace},  full descriptions of the vector spaces of left invariant CKY $2$-forms (including KY $2$-forms) are given, proving that they are one dimensional. 

%Finally, we consider the simply connected Lie groups $G$ endowed with a left invariant metric $g$ associated with the metric Lie algebras analyzed in Section $3$ and Section $5$. We  study their holonomy group, and we prove in Theorem \ref{irreducible} and Theorem \ref{irreducible_2} that they are irreducible as Riemannian manifolds. Moreover, in Section $7$ we exhibit some computations used in Section $6$.

\

\begin{small}
	\noindent{\itshape Acknowledgement.}
	Part of this work has been set up during the visit of the first-named author at KU Leuven, she thanks the Mathematics department at Campus Kulak Kortrijk for their hospitality.
	The authors would like to thank Adri\'an Andrada for several interesting discussions on the subject, and we are also grateful to the anonymous referee for carefully reading the manuscript
	and her/his suggestions that improved the exposition of the paper.
\end{small}

\section{Left invariant CKY 2-forms on Lie groups}
\subsection{Left invariant structures on Lie groups} 
Let $G$ be a Lie group and $\ggo$ its Lie algebra of left invariant vector fields. It is known that there is a linear isomorphism between $\ggo$ and $T_eG$ where $e$ is the identity in $G$. 
We denote by $g$ a left invariant Riemannian metric on $G$ satisfying ${L_a}^*g=g$, where $L_a$ denotes the left translation by $a\in G$. There is a well-known correspondence between left invariant metrics on $G$ and inner products on $\ggo \cong T_eG$ defined as $\la \cdot,\cdot\ra:=g_e(\cdot,\cdot)$. 

Every left invariant metric $g$  defines a unique Levi-Civita connection $\nabla$ on $G$, and  for every $x,y,z\in \ggo$ it has the following simple expression
\begin{equation}\label{koszul}
2\la \nabla_xy,z\ra=\la [x,y],z\ra-\la [y,z],x\ra+\la [z,x],y\ra.
\end{equation}
In particular, it is easy to see that for any $x\in\ggo$, the endomorphism $\nabla_x:\ggo\to \ggo$ is skew-symmetric with respect to $\la \cdot, \cdot \ra$.

A $2$-form $\omega$ is left invariant if $L_a^* \omega=\omega$ for all $a\in G$, and every left invariant $2$-form $\omega$ determines a skew-symmetric bilinear form on $\ggo$, which we also denote by $\omega$. 
Reciprocally, every skew-symmetric bilinear form on $\ggo$ defines a left invariant 2-form on $G$. 

There is a left invariant skew-symmetric $(1,1)$ tensor $T:TG\to TG$ associated to each left invariant $2$-form $\omega$, such that, $\omega(X,Y)=g(TX,Y)$ for all $X,Y$ vector fields on $G$. In particular, this induces a skew-symmetric endomorphism on $\ggo$, which we still denote by $T$. Reciprocally, every skew-symmetric endomorphism $T$ on $\ggo$ defines a skew-symmetric bilinear form $\omega$ on $\ggo$, and as a consequence induces a left invariant 2-form on $G$.
We will study left invariant $2$-forms on $G$ satisfying condition \eqref{ckyManifolds}, and the associated endomorphism $T$ will be called a CKY tensor.  We will make no distinction between the CKY $2$-form $\omega$ and the associated CKY tensor $T$.

\subsection{CKY tensors on Lie algebras} 
In \cite{Semmelmann}, Semmelmann gives an equivalent definition for CKY $2$-forms on Riemannian manifolds $(M,g)$, that is,  $\omega$ is a CKY $2$-form on $(M,g)$ if and only if there exists a $1$-form $\theta$  such that %for any 
\begin{equation}\label{cky_equation_w}
(\nabla_X \omega)(Y,Z)+(\nabla_Y \omega)(X,Z)=2g( X,Y) \theta(Z)-g( Y,Z) \theta(X)-g( X,Z)\theta(Y)
\end{equation} 
for all $X,Y,Z$ vector fields on $M$. If such a $1$-form exists, then it is uniquely determined by $\omega$ and it is given by $$\theta(X)=-\displaystyle{\frac{1}{n-1}\mathrm{d} ^*\omega(X)
},$$
where $n$ denotes the dimension of $M$.
We consider now a Lie group $G$ equipped with a left invariant metric $g$, and let $\omega$ be a left invariant CKY $2$-form on $(G,g)$ with associated CKY tensor $T$. Since $\omega$ is left invariant, then $\theta$ is  left invariant as well. Moreover, equation \eqref{cky_equation_w} is equivalent to the following expression at the Lie algebra level
\begin{equation}
	\label{ckyequationliealgebra}
	\la \left(\nabla_x T\right)y,z\ra+\la \left(\nabla_yT\right)x,z\ra=2\la x,y\ra \theta(z)-\la y,z\ra \theta(x)-\la x,z\ra \theta(y),
\end{equation} 
for some $\theta\in\ggo^*$, and we will refer to \eqref{ckyequationliealgebra} as the CKY condition along this work. Note that \eqref{ckyequationliealgebra} is symmetric with respect to $x$ and $y$.
We denote by $\xi$ the unique element of $\ggo $ such that 
	\[\theta(x)=\la \xi,x\ra \]
for all $x\in \ggo$, and we will refer to $\xi$ as the vector associated to $T$.
We are interested in strict CKY tensors, that is, $\theta\neq 0$. %, which implies that $T\neq0$. 
We recall now some general results of CKY tensors from \cite{AD} that will be useful later.

\begin{prop}\label{xi_unique} If $T$ is a CKY  tensor on a metric Lie algebra $(\ggo, \la \cdot , \cdot \ra)$ with associated vector $\xi$, then $T\xi=0$.
\end{prop} 

\begin{thm}
	\label{ad1}
	Let $T$ be a CKY tensor on the metric Lie algebra $(\ggo, \la \cdot , \cdot \ra)$, with associated vector $\xi$. If $\xi\neq 0$, then $\dim \ggo$ is odd and $T | _{\xi^ {\perp}}  : {\xi}^ {\perp} \to {\xi}^ {\perp}$ is a linear isomorphism. Moreover,
	$(\nabla_{\xi} T )\xi = \nabla_{\xi} \xi=0$, and ${\xi}^ {\perp}$ is stable by the operator $\ad_\xi$.
\end{thm}

\begin{rem}\label{rescaling}
If $T$ is a CKY tensor on a given metric Lie algebra $(\ggo, \ip)$, and $\xi$ is the unique element such that $\theta (x)=\la \xi,x\ra$ for all $x\in \ggo$ as described above, then $\theta (\xi)=\| \xi \|^2$. 
If we consider $cT$ with $c\in \mathbb{R}^*$, then $cT$ is a CKY tensor as well, since it is easily verified that the space of CKY tensors is a vector space. 
Then, the $1$-form associated to $cT$ is given by $\widehat{\theta}=c\theta$. Moreover, the unique
element $\widehat{\xi}$ such that $\widehat{\theta}(x)=\la \widehat{\xi},x\ra$ is $\widehat{\xi}=c\xi$.
Therefore, without loss of generality, we may assume that $\xi$ is unitary, since we can replace $T$ by $\frac{1}{\| \xi \|}T$, if necessary. From now on, we will assume  $\| \xi \|=1$, in particular $\theta(\xi)=1$. 
\end{rem}

In \cite{AD} it was proven that if $\ggo$ is unimodular, then $\xi\in\ggo'$. Here, we consider the general case when $\ggo$ is not necessarily unimodular and we obtain the following result:

\begin{lem}\label{xiperpgg} 
	If $T$ is a strict CKY  tensor on a metric Lie algebra $(\ggo, \la \cdot , \cdot \ra)$ with associated vector $\xi$, then 	$\la \xi, \ggo'\ra\neq 0$. 
\end{lem}

\begin{proof}
	Let $\{e_i\}$ be an orthonormal basis of $\ggo$, then
	\begin{align*}
	\theta(\xi)&=-\displaystyle{\frac{1}{n-1}\mathrm{d}^*\omega(\xi)
	} 	=\displaystyle{\frac{1}{n-1}\sum_{i=1}^n(\nabla_{e_i} \omega) (e_i,\xi)} 
	=-\displaystyle{\frac{1}{n-1}\sum_{i=1}^n\omega(\nabla_{e_i}e_i,\xi)+\omega(e_i,\nabla_{e_i}\xi)} \\
	&=-\displaystyle{\frac{1}{n-1}\sum_{i=1}^n \la T\nabla_{e_i}e_i,\xi\ra+\la Te_i,\nabla_{e_i}\xi\ra}
	%& = \displaystyle{\frac{1}{2(n-1)}\sum_{i=1}^n \la[e_i,\xi], Te_i\ra-\la [\xi,Te_i],e_i\ra+\la[Te_i,e_i],\xi\ra }\\
	%&=\displaystyle{\frac{1}{2(n-1)}\sum_{i=1}^n -\la \ad_{\xi}e_i,Te_i\ra-\la \ad_{\xi}Te_i,e_i\ra+\la[Te_i,e_i],\xi\ra }\\	
	%&=\displaystyle{\frac{1}{2(n-1)}\sum_{i=1}^n \la T\ad_{\xi}e_i,e_i\ra-\la \ad_{\xi}Te_i,e_i\ra+\la[Te_i,e_i],\xi\ra }
	%&=\displaystyle{\frac{1}{2(n-1)}\sum_{i=1}^n \la [T,\ad_{\xi}]e_i,e_i\ra+\la[Te_i,e_i],\xi\ra }
	=-\displaystyle{\frac{1}{2(n-1)}\sum_{i=1}^n \la [Te_i,e_i],\xi\ra}.
	\end{align*}
	Since, $\theta(\xi)\neq0$, then $\la \xi, \ggo'\ra \neq 0$.
\end{proof}

As said in the introduction,
in \cite[Theorem $4.3$]{AD}, Andrada and Dotti proved  that a metric Lie algebra admitting a strict CKY tensor has odd dimension. Therefore, we look for CKY $2$-forms on odd dimensional metric Lie algebras. 

The first step was taken in \cite{ABD}, where all CKY $2$-forms in dimension $3$ were classified: the Lie algebras $\mathfrak{su}(2)$, $\mathfrak{sl}(2,\RR)$, $\mathfrak{h}_3$ and $\aff(\RR)\times \RR$ admit strict CKY 2-forms for certain metrics. We follow the notation in \cite{Mi} and in \cite{ABDO} where  $\mathfrak{su}(2)$ is the Lie algebra of the Lie group $SU(2)$, $\mathfrak{sl}(2,\RR)$ is the Lie algebra of the Lie group $SL(2,\RR)$, $\mathfrak{h}_3$ denotes the $3$-dimensional Heisenberg Lie algebra and $\mathfrak{aff}(\mathbb{R})$ denotes the 2-dimensional non-abelian Lie algebra.

%For higher dimension, it is known that the only  2-step nilpotent Lie algebra with a strict CKY 2-form  is $\mathfrak{h}_{2n+1}$, \cite[Theorem 4.4]{ABD} (see also \cite{dBM21}), where $\mathfrak{h}_{2n+1}$ denotes the Heisenberg Lie algebra of dimension $2n+1$.  Note that $\mathfrak{h}_{2n+1}$ admits Sasakian structures, moreover it is the only $(2n + 1)$-dimensional nilpotent Lie algebra admitting Sasakian structures according to \cite[Theorem $5$]{AFV}.

Therefore, the next natural step is to consider $5$-dimensional metric Lie algebras; we analyze them taking into account the dimension of their center. 
In Section $3$ we study the case with $1$-dimensional center, while in Section $5$ we focus on the Lie algebras with center of dimension greater that one.

\section{CKY tensors on $5$-dimensional Lie algebras with $\dim\zz=1$}

In this section we classify all $5$-dimensional metric Lie algebras 
admitting a strict CKY $2$-form whose co-differential lies in the dual of the center, in particular the center
is $1$-dimensional according to \cite[Theorem 4.6]{AD}.
In addition, we determine all possible strict CKY tensors on these metric Lie algebras (see Table \ref{tabla1}).

First, let us recall a result from \cite{AD}, where the authors characterize 
$n$-dimensional metric Lie algebras admitting  strict CKY tensors with associated vector $\xi$ in the center, as a central extension of  $(n-1)$-dimensional Lie algebras equipped with an invertible KY tensor. 
For the sake of completeness we include this result, see \cite[Theorem 4.6 and Theorem 4,8]{AD}.

\begin{thm}
	\label{teorema dotti-andrada}
	Let $S$ be an invertible KY tensor on the metric Lie algebra $(\mathfrak{h},[\cdot,\cdot],\langle \cdot,\cdot \rangle)$ such that the  $2$-form $\mu(x,y)=-2\la S^{-1}x,y\ra$ is closed. Set $\mathfrak{g}:=\mathfrak{h}\oplus_{\mu} \mathbb{R} \xi $ the central extension of $\mathfrak{h}$ by the $2$-form $\mu$, that is the vector space $\mathfrak{h}\oplus \mathbb{R} \xi $ equipped with the Lie bracket $[\cdot,\cdot]_{\mu}$ defined by
	\begin{equation}\label{extensioncon mu}
		[\mathfrak{h},\xi]_{\mu}=0, \ \ [x,y]_{\mu}= [x,y]+\mu(x,y)\xi,   \  x,y \in \mathfrak{h};
	\end{equation}
	the inner product on $\ggo$ is defined by extending the one on $\mathfrak{h}$ by $\langle \mathfrak{h},\xi\rangle =0$, with $\|\xi\|>0$ arbitrary. Then, the endomorphism $T$ of $\ggo$ given by  $T|_{\mathfrak{h}}=S$ and  $T\xi=0$ is a strict CKY tensor on $\mathfrak{g}$.
	
	Conversely, any metric Lie algebra  $(\ggo,\ip)$  admitting a strict CKY tensor $T$ with associated vector $\xi$ in the center is obtained  in this way, where $\mathfrak{h}=\xi^\perp$, $S:=T|_\mathfrak{h}$, and the Lie bracket on $\mathfrak{h}$ is the $\mathfrak{h}$-component of the Lie bracket on $\ggo$.
	Moreover, the center of $\ggo$ is generated by $\xi$. 
\end{thm}

Second, every KY tensor on a $4$-dimensional metric Lie algebra is parallel according to \cite[Lemma 3.7]{ABM}. Notice that any parallel skew-symmetric endomorphism is KY, so in particular if it is invertible it satisfies the hypothesis of Theorem \ref{teorema dotti-andrada} (see also \cite[Corollary 2.2]{AD}). 
More recently, in \cite[Section 3]{Herrera} a classification of all non-abelian $4$-dimensional metric Lie algebras $(\mathfrak h,\la \cdot, \cdot \ra)$ that carry parallel skew-symmetric endomorphisms was made. 
This classification is up to isometric isomorphisms and for each fixed metric Lie algebra all parallel skew-symmetric tensors are given up to equivalence, where,
two parallel skew-symmetric tensors $H_1$ and $H_2$ are said to be \textit{equivalent} if there exists an isometric isomorphism of Lie algebras 
\begin{equation}\label{equiv}
	\varphi: \ggo\to\ggo \; \text{such that} \; \varphi H_1=H_2 \varphi.	
\end{equation}

These results altogether allow us to give a full classification of $5$-dimensional metric Lie algebras admitting a strict CKY $2$-form whose co-differential lies in the dual of the center. 
Before we state the main result of this section we give an example of this construction.

\begin{ex} \label{Re2}
	Let $\mathfrak h$ be the $4$-dimensional Lie algebra $\mathbb R\times\mathfrak e(2)$ with Lie brackets given by \linebreak $[e_1,e_2]=-f_2$ and $[e_1,f_2]=e_2$. The metric $\ip_t$ is defined by $\text{diag}(t,t,t,t)$	in the basis $\{e_1,f_1,e_2,f_2\}$, for $t>0$. 
	According to \cite{Herrera}, any invertible parallel skew-symmetric endomorphism $H$ on $(\mathfrak h,\ip_t)$ is given by  
	$Se_1=f_1$, $Se_2=a_2f_2$
%	$\tiny S=\begin{pmatrix} 0&-{a_1}&0&0 \\ {a_1}&0&0&0\\ 0&0&0&-{a_2}\\ 0&0&a_2&0  \end{pmatrix}$ 
	for some non-zero $a_1,a_2\in \mathbb R$. Then, it follows from Theorem \ref{teorema dotti-andrada} and the discussion above that the central extension metric Lie algebra $(\RR \times \mathfrak e(2)\oplus_{\mu}\RR \xi,\ip_t)$ with $\|\xi\|=1$, $\mu(x,y)=-2\la S^{-1}x,y\ra$,  and Lie brackets given by
	$$[e_1,e_2]_{\mu}=-f_2, \; [e_1,f_2]_{\mu}= e_2,  \;
	[e_1,f_1]_{\mu}=2t{a_1}^{-1}\xi, \;  [e_2,f_2]_{\mu}=2t{a_2}^{-1}\xi$$
	admits a CKY tensor $T$ given by $T|_\mathfrak h =S$ and $T(\xi)=0$. In particular, the associated $2$-form is given by $\omega=a_1te^1\wedge f^1+a_2te^2\wedge f^2$, where $\{e^1,f^1,e^2,f^2\}$ is the metric dual basis of $\{e_1,f_1,e_2,f_2\}$. We denote  $\RR \times \mathfrak e(2)\oplus_{\mu}\RR \xi$ by $\ggo_{a_1,a_2,t}$ to emphasize that the structure constants depend on those parameters.
\end{ex}

Next, we have the main result of this section:

\begin{thm}\label{classificacion_1}
Let $(\ggo,\ip)$ be a $5$-dimensional metric Lie algebra and denote by $\zz$ its center. 
If $(\ggo,\ip)$ admits a strict CKY tensor $T$ with associated vector $\xi\in\zz$, then $(\ggo,\ip)$ is isometrically isomorphic to one and only one of the metric Lie algebras in Table \ref{tabla1}.

Moreover, the CKY tensor $T$ is uniquely determined, up to scaling, by the corresponding  metric Lie algebra and it 
%can be represented by the matrix 
is
given in the last column of Table \ref{tabla1}.	
\end{thm}

\begin{proof}
	Proceeding as in Example \ref{Re2} with all metric Lie algebras admitting an invertible parallel skew-symmetric tensor classified in \cite{Herrera}, we obtain all possible strict CKY tensors on a $5$-dimensional metric  Lie algebra with $\xi\in\zz$. 
	Moreover, $T$ is uniquely determined, up to scaling, by the corresponding  metric Lie algebra (see Remark \ref{rescaling}).

	It remains to show that each metric Lie algebras $(\ggo,\ip)$ in Table \ref{tabla1} are pairwise non-isometrically isomorphic. Clearly, elements from different families (rows) are not isomorphic. 
	Now we check metric Lie algebras within the same family. Since they are similar, we only explain one of those families.

	Consider the metric Lie algebra from Example \ref{Re2} with parameters $a_1,a_2,t\in\mathbb{R}_{>0}$, and assume $\Phi:(\ggo_{a_1,a_2,t},\ip_t)\to(\ggo_{a'_1,a'_2,t'},\ip_{t'})$ is an isometric isomorphism. First, we show that $t=t'$. Indeed, note that $\Phi(\xi)=\pm\xi$, and $\la \Phi(x),\xi\ra_{t'}=0$ for all $x\in \mathfrak{h}=\xi^{\perp}$, thus $\Phi|_{\xi^{\perp}}$ preserves $\xi^{\perp}$.
	Moreover,
	$\Phi|_{\mathfrak{h}}: (\mathfrak{h},[\cdot , \cdot]_{\mathfrak{h}\times\mathfrak{h}},\ip_{t}) \to (\mathfrak{h}, [\cdot , \cdot]_{\mathfrak{h}\times\mathfrak{h}},\ip_{t'})$ is an isometric automorphism as well.
	According to \cite[Proposition 3.6]{Herrera} we have that $(\mathfrak{h},[ \cdot, \cdot]_{\mathfrak{h}\times\mathfrak{h}},\ip_{t})$ is isometrically isomorphic to $(\mathfrak{h},[ \cdot, \cdot]_{\mathfrak{h}\times\mathfrak{h}},\ip_{t'})$ if and only of $t=t'$.
	
	Now we fix $t$ and we show that $(a_1,a_2)=(a'_1,a'_2)$. Using \eqref{extensioncon mu} and the fact that $\Phi|_{\mathfrak{h}}$ is an isomorphism, we have that $\mu'(\Phi(x),\Phi(y))=\mu(x,y)$, or equivalently, $S'=\Phi^tS\Phi$. Thus, $S$ and $S'$ are equivalent according to \eqref{equiv}. Again, according to \cite[Proposition 3.6]{Herrera}, $S$ is equivalent to $S'$ if and only if $(a_1,a_2)=(a'_1,a'_2)$.	
\end{proof}

\begin{table}[h!]
		\begin{tabular}{l l c l}
			\hline
			\hline
			
			 {\scriptsize  }
		 & {\scriptsize Metric Lie algebra} && {\scriptsize strict CKY $2$-form $\omega$}  \\
			\hline \hline
		\\
		\begin{tiny}
			$ (\RR \times \mathfrak e(2)\oplus_{\mu}\RR \xi,\ip_t)$
		\end{tiny} & 
	  \begin{tiny} 
	  	$\begin{matrix} 
			[e_1,e_2]_{\mu}=-f_2, \\
			[e_1,f_2]_{\mu}= e_2,            \\  %&\|e_i\|^2=\|f_i\|^2=t, 
			[e_1,f_1]_{\mu}=2t{a_1}^{-1}\xi  \\  %&\|\xi\|=1,
			[e_2,f_2]_{\mu}=2t{a_2}^{-1}\xi, \\  %&a_1,a_2,t> 0,
			\end{matrix} $
		\end{tiny} 
		&
		{\tiny $\begin{matrix} 
		\|e_i\|^2=\|f_i\|^2=t, \\
		\|\xi\|=1,\\
		a_1,a_2,t> 0,
    	\end{matrix} $}
			
		& {\scriptsize $ta_1e^1\wedge f^1+ta_2e^2\wedge f^2$}\\
		\\\hline 
	
%	&	& {\scriptsize	$ \begin{pmatrix} 0&-{a_1}&0&0&0 \\ {a_1}&0&0&0&0 \\ 0&0&0&-{a_2}&0\\ 0&0&a_2&0&0 \\ 0&0&0&0&0 \end{pmatrix}$ }	\\			

		\\
		\begin{tiny}
		$(\RR^2\times\aff(\RR) \oplus_{\mu}\RR \xi,\ip_t)$\end{tiny} & 			\begin{tiny}$ 
		\begin{matrix} [e_1,f_1]_{\mu}=2t{a_1}^{-1} \xi, \\
		 [e_2,f_2]_{\mu}= f_2+2t{a_2}^{-1}\xi,  \\
		\end{matrix}$	\end{tiny}
	&
	    {\tiny $\begin{matrix} 
	    \|e_i\|^2=\|f_i\|^2=t, \\	
	    \|\xi\|=1,\\
	    a_1,a_2,t>0, 
	    \end{matrix}$ }
		 & {\scriptsize $ta_1e^1\wedge f^1+ta_2e^2\wedge f^2$ }\\
		%{\tiny $ \begin{pmatrix} 0&-{a_1}&0&0&0 \\ {a_1}&0&0&0&0 \\ 0&0&0&-{a_2}&0\\ 0&0&a_2&0&0 \\ 0&0&0&0&0 \end{pmatrix}$}
			\\ \hline
			\\
			{\tiny $(\mathfrak{r}'_{4,\lambda,0}\oplus_{\mu}\RR \xi,\ip_t)$}& \begin{tiny}
			$\begin{matrix} [e_1,f_1]_{\mu}=\lambda f_1+2t{a_1}^{-1}\xi, \\  [e_1,f_2]_{\mu}=-e_2,\\ 
			[e_1,e_2]_{\mu}= f_2, \\ 
			[e_2,f_2]_{\mu}=2t{a_2}^{-1}\xi, 
            \end{matrix}$\end{tiny}
        &
        {\tiny $\begin{matrix}
        	\|e_i\|^2=\|f_i\|^2=t,\\
        	 \|\xi\|=1,\\
        	 \lambda,a_1,a_2,t>0,
        \end{matrix}$	}
			 & {\scriptsize $ta_1e^1\wedge f^1+ta_2e^2\wedge f^2$}\\
		%{\tiny $ \begin{pmatrix} 0&-{a_1}&0&0&0 \\ {a_1}&0&0&0&0 \\ 0&0&0&-{a_2}&0\\ 0&0&a_2&0&0 \\ 0&0&0&0&0 \end{pmatrix}$}
			\\ \hline
			\\
			
		{\tiny $(\aff(\RR)\times \aff(\RR)\oplus_{\mu}\RR \xi,\ip_{t,s})$}& \begin{tiny}
		$\begin{matrix} [e_1,f_1]_{\mu}= f_1+2t{a_1}^{-1}\xi,\\
		[e_2,f_2]_{\mu}= f_2+2ts{a_2}^{-1}\xi,\\  
		\end{matrix}$\end{tiny}
	&
		{\tiny $\begin{matrix}
		\|e_1\|^2=\|f_1\|^2=t, \\
		\| e_2\|^2= \|f_2\|^2=ts,\\		
		\|\xi\|=1, t,s>0, s \leq 1 \\
		a_1,a_2> 0  \mbox{ if } s<1, \mbox{ or } \\
		a_1\geq a_2> 0  \mbox{ if } s=1
		\end{matrix}$	}
			 & {\scriptsize $ta_1e^1\wedge f^1+sta_2e^2\wedge f^2$}\\
		%{\tiny $ \begin{pmatrix} 0&-{a_1}&0&0&0 \\ {a_1}&0&0&0&0 \\ 0&0&0&-{a_2}&0\\ 0&0&a_2&0&0 \\ 0&0&0&0&0 \end{pmatrix}$}
		\\\hline
		\\
		
		{\tiny $(\mathfrak{d}_{4,\frac{1}{2}}\oplus_{\mu} \RR \xi,\ip_t)$}& \begin{tiny}
		$\begin{matrix} [e_1,f_1]_{\mu}=f_1+2tc^{-1}\xi, \\ [e_1,e_2]_{\mu}=\frac{1}{2}e_2, \\
		[e_1,f_2]_{\mu}=\frac{1}{2}f_2, \\
		[e_2,f_2]_{\mu}=f_1+2tc^{-1}\xi \\
		\end{matrix}$ \end{tiny}
	&
		{\tiny $\begin{matrix}
		 \|e_i\|^2=\|f_i\|^2=t,\\
		  \|\xi\|=1\\
		 c,t> 0, &
		\end{matrix}$ }
		 & {\scriptsize $tc(e^1\wedge f^1+e^2\wedge f^2)$}\\
%{\tiny $c\begin{pmatrix} 0&-1&0&0&0 \\ 1&0&0&0&0 \\ 0&0&0&-1&0\\ 0&0&1&0&0 \\ 0&0&0&0&0 \end{pmatrix}$}
		\\ \hline
		\\

		{\tiny $(\mathfrak{d}_{4,2}\oplus_{\mu} \RR \xi,\ip_t)$}& \begin{tiny}
		$\begin{matrix} [e_1,f_1]_{\mu}=-e_1+2tc^{-1}\xi, \\
		[e_1,e_2]_{\mu}=f_2,\\ 
		[f_1,e_2]_{\mu}=-\frac{1}{2}e_2, \\
		 [f_1,f_2]_{\mu}=\frac{1}{2}f_2\\
		[e_2,f_2]_{\mu}= 2tc^{-1}\xi\\
	    \end{matrix}$ \end{tiny}
    &
		{\tiny $\begin{matrix}
			\|e_i\|^2=\|f_i\|^2=t ,\\
			\|\xi\|=1 \\
			c,t> 0
		\end{matrix}$ }
		 & {\scriptsize $tc(e^1\wedge f^1+e^2\wedge f^2)$}\\
%{\tiny $ c\begin{pmatrix} 0&-{1}&0&0&0 \\ {1}&0&0&0&0 \\ 0&0&0&-{1}&0\\ 0&0&1&0&0 \\ 0&0&0&0&0 \end{pmatrix}$}  
		\\ \hline
		\\
		
		{\tiny $(\mathfrak{d}'_{4,\frac\delta2}\oplus_{\mu} \RR \xi,\ip_t)$ }& \begin{tiny}
			$\begin{matrix} 
			[e_1,f_1]_{\mu}=f_1+2tc^{-1}\xi, \\ [e_1,e_2]_{\mu}=\frac{1}{2}e_2-\frac{1}{\delta}f_2\\
			[e_1,f_2]_{\mu}=\frac{1}{\delta}e_2+\frac{1}{2}f_2, \\ [e_2,f_2]_{\mu}=f_1+2tc^{-1}\xi
            \end{matrix}$ \end{tiny} 
    &
		{\tiny $\begin{matrix} 
				\|e_i\|^2=\|f_i\|^2=t ,\\
				\|\xi\|=1,\\
				\delta>0,c\neq 0
		\end{matrix}$	}
			& {\scriptsize $tc(e^1\wedge f^1+e^2\wedge f^2)$}\\
	%	{\tiny $ c\begin{pmatrix} 0&-{1}&0&0&0 \\ {1}&0&0&0&0 \\ 0&0&0&-{1}&0\\ 0&0&1&0&0 \\ 0&0&0&0&0 \end{pmatrix}$} 
		    \\ \hline
		    \\
		    
			{\tiny $(\mathfrak{\RR}^4 \oplus_{\mu}\RR \xi,\ip)$}& \begin{tiny} $ 
			\begin{matrix} [e_1,f_1]_{\mu}=2{a_1}^{-1} \xi, \\ [e_2,f_2]_{\mu}=2{a_2}^{-1}\xi,  
			\end{matrix}$ \end{tiny} 
	&
		   {\tiny  $\begin{matrix} 
		  	\|e_i\|=\|f_i\|= \|\xi\|=1,\\
		  	a_1,a_2> 0
		    \end{matrix}$ }
			& {\scriptsize $a_1e^1\wedge f^1+a_2e^2\wedge f^2$}\\
		
		%{\tiny $ \begin{pmatrix} 0&-{a_1}&0&0&0 \\ {a_1}&0&0&0&0 \\ 0&0&0&-{a_2}&0\\ 0&0&a_2&0&0 \\ 0&0&0&0&0 \end{pmatrix}$} 
	    \\ \hline
		\end{tabular}
	\caption{ {\scriptsize The strict CKY $2$-forms %tensors $T$ 
		are given in the metric dual basis of $1$-forms
		%$\{e_1,f_1,e_2,f_2,\xi\}$. 
		$\{e^1,f^1,e^2,f^2\}$ } }
	\label{tabla1}
\end{table} 

Now we unify our notation with  \cite{AFV}. The Lie algebra $\RR \times \mathfrak e(2)\oplus_{\mu}\RR \xi$ in Table \ref{tabla1} corresponds to $\ggo_3$ in the notation of \cite{AFV}, similarly $\RR^2\times \aff(\RR) \oplus_{\mu}\RR\xi\cong \ggo_2$, $\mathfrak{r}'_{4,\lambda,0}\oplus_{\mu_2}\RR \xi \cong \mathfrak{g}_8^{1/\lambda}$, $\aff(\RR)\times \aff(\RR)\oplus_{\mu}\RR \xi \cong \ggo_4$,
$\mathfrak{d}_{4,\frac{1}{2}}\oplus_{\mu} \RR \xi \cong \ggo_5$, $\mathfrak{d}_{4,2}\oplus_{\mu} \RR \xi \cong \ggo_6$,  $\mathfrak{d}'_{4,\frac\delta2}\oplus_{\mu} \RR \xi \cong \ggo_7^{\delta}$ and $(\RR^4 \oplus_{\mu} \RR \xi,\ip)\cong \mathfrak{h}_5$. In Table \ref{NotacionNueva} we exhibit the metric Lie algebras and the strict CKY tensor with the new notation.

{\tiny
\begin{table}[h!]
	\begin{tabular}{c c}
		\hline \hline
		{\scriptsize Metric Lie algebra }  &  {\scriptsize strict CKY $2$-form $\omega$} \\ \hline \hline 
		\begin{tiny}
			$\begin{matrix} (\mathfrak{g}_3,\ip_{t,a_1,a_2})\\ 
				&\\
				[E_1,E_2]=-E_3\\
				[E_1,E_3]=E_2\\
				[E_1,E_4]=-E_5\\
				[E_2,E_3]=-E_5\end{matrix}$
		\end{tiny}
	\begin{tiny}  
		 $\la \cdot,\cdot \ra_{t,a_1,a_2}=\begin{pmatrix} 
				t & 0 & 0 & 0 & 0 \\
				0&t& 0&0 & 0\\
				0&0&t&0&0 \\
				0&0& 0& \frac{a_1^2t}{a_2^2}&0 \\
				0&0& 0& 0&\frac{4t^2}{a_2^2} \\
			\end{pmatrix}$\end{tiny}
%		&\begin{tiny}
%			$\begin{pmatrix} 
%				0 & 0 & 0 & -\frac{a_1^2}{a_2} & 0 \\
%				0&0& -a_2&0&0 \\
%				0&a_2&0&0&0 \\
%				a_2&0& 0& 0&0 \\
%				0&0& 0& 0&0 \\
%			\end{pmatrix}$
%		\end{tiny} 		\\
& $\frac{a_1^2 t}{a_2}e^{14}+ a_2te^{23}$ \\
		 \hline 
		\begin{tiny}
			$\begin{matrix} (\mathfrak{g}_2,\ip_{t,a_1,a_2})\\ 
				&\\
				[E_1,E_2]=E_2-E_5\\
				[E_3,E_4]=-E_5\end{matrix}$
		\end{tiny}
		\begin{tiny}  
		 $\la \cdot,\cdot \ra_{t,a_1,a_2}=\begin{pmatrix} 
				t & 0 & 0 & 0 & 0 \\
				0&t& 0&0 \\
				0&0&a_1^2t&0&0 \\
				0&0& 0& \frac{t}{a_2^2}&0 \\
				0&0& 0& 0&\frac{4t^2}{a_2^2} \\
			\end{pmatrix}$\end{tiny} 
%		&\begin{tiny}
%			$\begin{pmatrix} 
%				0 & -a_2 & 0 & 0 & 0 \\
%				a_2&0& 0&0&0 \\
%				0&0&0&-\frac{1}{a_2}&0 \\
%				0&0& a_2a_1^2& 0&0 \\
%				0&0& 0& 0&0 \\
%			\end{pmatrix}$
%		\end{tiny} 	 \\ 
	& $ta_2e^{12}+ \frac{ta_1^2}{a_2}e^{34}$ \\
	     \hline 
	     
	     \begin{tiny}
	     	$\begin{matrix} (\mathfrak{g}_8^{\frac1\lambda},\ip_{t,a_1,a_2})\\ 
	     		&\\
	     		[E_1,E_4]=-E_1-E_5\\
	     		[E_2,E_3]=-E_5\\
	     		[E_2,E_4]=\frac1\lambda E_3\\
     			[E_3,E_4]=-\frac1\lambda E_2\end{matrix}$
	     \end{tiny}
	     \begin{tiny}  
	     	$\la \cdot,\cdot \ra_{t,a_1,a_2}=\begin{pmatrix} 
	     		(\frac{a_1\lambda }{a_2})^2t & 0 & 0 & 0 & 0 \\
	     		0&t& 0&0 \\
	     		0&0&t&0&0 \\
	     		0&0& 0& \frac{t}{\lambda^2}&0 \\
	     		0&0& 0& 0&\frac{4t^2}{a_2^2} \\
	     	\end{pmatrix}$\end{tiny}
%	     &\begin{tiny}
%	     	$\begin{pmatrix} 
%	     		0 & 0 & 0 & \frac{a_2}{\lambda^2} & 0 \\
%	     		0&0& a_2&0&0 \\
%	     		0&-a_2&0&0&0 \\
%	     		-\frac{(a_1\lambda)^2}{a_2}&0& 0& 0&0 \\
%	     		0&0& 0& 0&0 \\
%	     	\end{pmatrix}$
%	     \end{tiny} 		\\
     & $-\frac{a_1^2 t}{a_2}e^{14}- a_2te^{23}$ \\
	     \hline 
	     
	     	     \begin{tiny}
	     	$\begin{matrix} (\mathfrak{g}_4,\ip_{s,t,a_1,a_2}) \\ 
	     		&\\
	     		[E_1,E_2]=E_2-E_5\\
	     		[E_3,E_4]=E_4-E_5\end{matrix}$
	     \end{tiny} 
	     \begin{tiny}  
	     	$\la \cdot,\cdot \ra_{s,t,a_1,a_2}=\begin{pmatrix} 
	     		t & 0 & 0 & 0 & 0 \\
	     		0&ta_1^2& 0&0&0 \\
	     		0&0&ts&0&0 \\
	     		0&0& 0& \frac{ta_2^2}{s}&0 \\
	     		0&0& 0& 0& 4t^2 \\
	     	\end{pmatrix}$\end{tiny}
%	     &\begin{tiny} 
%	     	$\begin{pmatrix} 
%	     		0 & -a_1^2 & 0 & 0 & 0 \\
%	     		1&0& 0&0&0 \\
%	     		0&0&0&-\frac{a_2^2}{s}&0 \\
%	         	0&0& s & 0&0 \\
%	     		0&0& 0& 0&0 \\
%	     	\end{pmatrix}$
%	     \end{tiny} 		\\
     	& $ta_1^2e^{12}+ ta_2^2e^{34}$ \\
	     \hline 
	    
		\begin{tiny}
			$\begin{matrix} (\mathfrak{g}_5,\ip_{t,c})\\ 
				&\\
				[E_1,E_2]=E_3-E_5\\
				[E_1,E_4]=-\frac{1}{2}E_1\\
				[E_2,E_4]=-\frac{1}{2}E_2\\
				[E_3,E_4]=-E_3+E_5\end{matrix}$
		\end{tiny}
		\begin{tiny}  
		 $\la \cdot,\cdot \ra_{t,c}=\begin{pmatrix} 
				t & 0 & 0 & 0 & 0 \\
				0&t& 0&0&0 \\
				0&0&t&0&0 \\
				0&0& 0& t&0 \\
				0&0& 0& 0&\frac{4t^2}{c^2} \\
			\end{pmatrix}$\end{tiny}
%		 &\begin{tiny}
%			$c\begin{pmatrix} 0&-1&0&0&0 \\ 1&0&0&0&0 \\ 0&0&0&1&0\\ 0&0&-1&0&0 \\ 0&0&0&0&0 \end{pmatrix}$
%		\end{tiny} \\ 
		& $tc(e^{12}-e^{34})$ \\
	\hline
	
		     	     \begin{tiny}
		$\begin{matrix} (\mathfrak{g}_6,\ip_{t,a_1,a_2}) \\ 
			&\\
			[E_1,E_2]=E_3\\
			[E_1,E_4]=-2E_1\\
			[E_2,E_3]=-E5\\
			[E_2,E_4]=E_2\\
			[E_3,E_4]=-E_3\end{matrix}$
	\end{tiny}
	\begin{tiny}  
		$\la \cdot,\cdot \ra_{t,c}=\begin{pmatrix} 
			t+\frac{4t^2}{c^2} & 0 & 0 & 0 & \frac{4t^2}{c^2} \\
			0&t& 0&0 &0\\
			0&0&t&0&0 \\
			0&0& 0& 4t&0 \\
			\frac{4t^2}{c^2}&0& 0& 0&\frac{4t^2}{c^2} \\
		\end{pmatrix}$\end{tiny}
%	&\begin{tiny}
%		$c\begin{pmatrix} 
%			0 & 0 & 0 & -2 & 0 \\
%			0&0& -1&0&0 \\
%			0&1&0&0&0 \\
%			\frac12&0& 0& 0&0 \\
%			0&0& 0& 2&0 \\
%		\end{pmatrix}$
%	\end{tiny} 		\\
		& $tc(2e^{14}+e^{23})$ \\
	\hline 
	
		\begin{tiny}
			$\begin{matrix} (\mathfrak{g}_7^{\delta},\ip_{t,c})\\ 
				&\\
				[E_1,E_2]=E_3-E_5\\
				[E_1,E_4]=-\frac\delta2 E_1+E_2\\
				[E_2,E_4]=-E_1-\frac\delta2 E_2 \\
				[E_3,E_4]=-\delta E_3+\delta E_5\end{matrix}$
		\end{tiny}
		\begin{tiny}  
		 $\la \cdot,\cdot \ra_{t,c}=\begin{pmatrix} 
				t & 0 & 0 & 0 & 0 \\
				0&t& 0&0&0 \\
				0&0&t&0&0 \\
				0&0& 0& \delta^2t &0 \\
				0&0& 0& 0&\frac{4t^2}{c^2} \\
			\end{pmatrix}$\end{tiny} 
%		&\begin{tiny}
%			$c\begin{pmatrix} 0&-1&0&0&0 \\ 1&0&0&0&0 \\ 0&0&0&2\lambda&0\\ 0&0&-\frac{1}{2\lambda}&0&0 \\ 0&0&0&0&0 \end{pmatrix}$
%		\end{tiny} \\ 
			& $tc(e^{12}-\delta e^{34})$ \\
	\hline
	
			\begin{tiny}
		$\begin{matrix} (\mathfrak{h}_5,\ip_{a_1,a_2})\\ 
			&\\
			[E_1,E_2]=E_5\\
			[E_3,E_4]=E_5 \end{matrix}$
	\end{tiny}
	\begin{tiny}  
		$\la \cdot,\cdot \ra_{a_1,a_2}=\begin{pmatrix} 
			a_1^2 & 0 & 0 & 0 & 0 \\
			0&\frac{1}{a_2^2}& 0&0&0 \\
			0&0&1&0&0 \\
			0&0& 0& 1&0 \\
			0&0& 0& 0&\frac{4}{a_2^2} \\
		\end{pmatrix}$\end{tiny}
%	 &\begin{tiny}
%		$\begin{pmatrix} 0&-\frac{1}{a_2}&0&0&0 \\ 
%						a_1^2a_2&0&0&0&0 \\ 
%						0&0&0&-a_2&0\\ 
%						0&0&a_2&0&0 \\ 
%						0&0&0&0&0 \end{pmatrix}$
%	\end{tiny} \\ 
		& $\frac{a_1^2}{a_2}e^{12}+a_2e^{34} $ \\
	\hline

	\end{tabular}
	\caption{ {\scriptsize The metrics are given in the basis $\{E_1,E_2,E_3,E_4,E_5\}$ and the strict CKY $2$-forms are given in the metric dual basis $\{e^1,e^2,e^3,e^4,e^5\}$ where $e^{ij}=e^i\wedge e^j$. For each metric Lie algebra, the parameters $s,t,a_1,a_2,c,\delta,\lambda$ satisfy the conditions of the corresponding metric Lie algebra in Table \ref{tabla1}. } }
	\label{NotacionNueva}
\end{table}
}

\begin{rem}
In \cite[Section 3.1]{AFV}, central extensions were considered to  
classify five dimensional Sasakian Lie algebras. More precisely, they use the classification of $4$-dimensional K\"ahler Lie algebras $(\mathfrak h,\la \cdot, \cdot \ra,J)$ given in \cite{Ov}. Since in \cite{Herrera} the author classifies all  parallel  skew-symmetric tensors (not only complex structures), we can recover the classification of Sasakian five dimensional Lie algebras.
%Indeed, for $a_1=a_2=1$ or $c=1$ in every line in Table \ref{tabla1}, we obtain the Sasakian structures exhibited in \cite{AFV}. In that work, the authors also show that the simply connected Lie group associated to the unimodular Lie algebra $\ggo_3$ admits lattices, that is, co-compact discrete subgroups. The other unimodular Lie algebra in Table \ref{tabla1} is $\mathfrak h_5=\mathfrak{\RR}^4 \oplus_{\mu}\RR \xi$, and it is known that the Heisenberg Lie group $H_5$ with Lie algebra 	$\mathfrak h_5$ also admits lattices.
Indeed, these Sasakian structures are obtained when the parallel skew-symmetric tensors on $\mathfrak h$ are in addition complex structures (for example the Sasakian structure on $\ggo_3$ exhibited in \cite{AFV} is obtained with $a_1=a_2=1$ in the first row of Table \ref{tabla1}). 
In \cite{AFV}, the authors also show that the simply connected Lie group associated to the unimodular Lie algebra $\ggo_3$ admits lattices, that is, co-compact discrete subgroups. The other unimodular Lie algebra in Table \ref{NotacionNueva} is $\mathfrak h_5=\mathfrak{\RR}^4 \oplus_{\mu}\RR \xi$, and it is known that the Heisenberg Lie group $H_5$ with Lie algebra 	$\mathfrak h_5$ also admits lattices.
\end{rem}

\section{CKY tensors on Lie algebras with $\dim\zz>1$}

In this section we study Lie algebras with $\dim\zz>1$ admitting CKY tensors, where $\zz$ denotes the center of the Lie algebra $\ggo$. 
In particular, such a Lie algebra cannot admit any Sasakian structure, since the center of a Sasakian Lie algebra is trivial  or has dimension $1$. As a consequence of this condition on the center of the Lie algebra, we have the following result.

\begin{lem}\label{centro}
	Let $T$ be a strict CKY tensor on the metric Lie algebra $(\ggo, \la \cdot , \cdot \ra)$, with associated 1-form $\theta$ and dual vector $\xi$. If $\on{dim}\zz\geq 2$, then $\theta(z)=0$ for all $z\in \mathfrak{z}$. In particular $\xi \perp \zz$.
\end{lem}

\begin{proof}
	Let $z_1,z_2\in \mathfrak{z}$,  such that $z_1,z_2\neq0$ and $\la z_1,z_2\ra=0$. From \eqref{ckyequationliealgebra} applied to $z_1,z_2$ we have
	\begin{align*}
	2\langle \left( \nabla_{z_1} T \right) z_1,z_2 \rangle &=2\langle z_1,z_1\rangle \theta(z_2)-\langle z_1,z_2\rangle \theta(z_1)-\langle z_1,z_2\rangle \theta(z_1).
	\end{align*}
	Since $\langle \left( \nabla_{z_1} T \right) z_1,z_2 \rangle=\langle  \nabla_{z_1} T  z_1,z_2 \rangle+\langle \nabla_{z_1}   z_1,Tz_2 \rangle =0$,
	we obtain that $\theta(z_2)=0$. 
\end{proof}

As a consequence of $\xi \in \zz^{\perp}$ we have the following condition:

\begin{lem}
	\label{zperpgg}
	 If $T$ is a strict CKY tensor on the metric Lie algebra $(\ggo, \la \cdot , \cdot \ra)$, then $T$ does not preserve the decomposition $\ggo=\zz\oplus \zz^{\perp}$.
\end{lem}

\begin{proof}
	Let $\theta$ and $\xi$ be the $1$-form  and dual vector associated to $T$. Suppose $z$ is a unitary vector on the $\zz$, then from 
	\eqref{ckyequationliealgebra} we obtain $$2\la (\nabla_z T)z,\xi\ra=2\|z\|^2\theta(\xi)=2,$$ 
	which implies $\la \nabla_zTz,\xi\ra=1$.
%	\begin{align*}
%	2\la (\nabla_z T)z,\xi\ra&=2\|z\|^2\theta(\xi)=2\\
%	\la \nabla_zTz,\xi\ra&=1
%	\end{align*}
	Then, it follows from equation \eqref{koszul} that $\la [Tz,\xi],z\ra=-2$, and therefore $T\zz\nsubseteq\zz$.
\end{proof}

%\begin{rem}\label{T_no_preserva_z}
%As a consequence of Lemma \ref{zperpgg}, we have that $T\zz\nsubseteq\zz$, and therefore, $T$ does not preserve the decomposition $\ggo=\zz\oplus \zz^{\perp}$. Moreover, it follows from Lemma \ref{xiperpgg} that $\la \zz,\ggo'\ra \neq 0$. {\color{blue} La primera oracion podria incorporarse al lema. La segunda oracion esta bien?} 
%\color{red}Como $\la [Tz,\xi],z\ra=-2$ y $[Tz,\xi]\in \ggo'$, entonces $\la z,\ggo'\ra\neq 0$.
%\end{rem}

In the next lemma we apply  \eqref{ckyequationliealgebra} to different choices of $x,y,z\in\ggo$, according to the decomposition $\ggo=\zz\oplus\zz^\perp$. 
We summarize the conditions that a Lie algebra must satisfy in order to admit a strict CKY tensor. These conditions will be useful in the next section. Note that if we have three central elements, then both sides of CKY condition \eqref{ckyequationliealgebra} are identically zero, since $\zz \perp \xi$.

\begin{lem}\label{lema21}
	If $(\ggo, \la \cdot , \cdot \ra)$ is a metric Lie algebra with  $\on{dim}\zz \geq2$ that admits a strict CKY tensor $T$  with associated $1$-form $\theta$ and dual vector $\xi$, then we have:
	\begin{itemize}
		\item[\ri] for $ z_1,z_2\in \zz, x\in \zz^\perp$,		
		\begin{align}
			\label{lema21-1}
%			\langle [x,Tz_1],z_2\rangle+2\langle [Tz_2,x],z_1\rangle &=-2\langle z_1,z_2\rangle \theta(x),\\  
%			\label{lema21-2}
%			\langle [x,Tz_1],z_2\rangle+\langle [x,Tz_2],z_1\rangle &=4\langle z_1,z_2\rangle \theta(x), 
			\langle [x,Tz_1],z_2\rangle=\langle [x,Tz_2],z_1\rangle=2\langle z_1,z_2\rangle \theta(x),
		\end{align}
		\item[\rii] for $ z\in \zz, x,y\in \zz^\perp$,
		\begin{align}
			\label{xyz1}
			\langle [x,Ty],z\rangle  +2\langle [Tz,x],y\rangle +2\langle [Tz,y],x\rangle +\langle [y,Tx],z\rangle &=0,\\  
			\label{xyz2}
			-\langle [Tz,x],y\rangle -\langle [Tz,y],x\rangle +\langle [y,x],Tz\rangle+2\langle [Ty,x],z\rangle-\langle [Tx,y],z\rangle  &=0,
		\end{align} 
		\item[\riii] for $x,y\in \zz^\perp$ with $x,y\perp\xi$,
		\begin{align}
			\label{031}
			\langle [\xi,Tx],y\rangle-\langle [Tx,y],\xi\rangle +\langle [y,\xi],Tx\rangle +2\langle [Ty,\xi],x\rangle  +2\langle [Ty,x],\xi\rangle=&-2\langle x,y\rangle,\\
			\label{032}
			\langle [x,Ty],\xi\rangle-\langle [Ty,\xi],x\rangle+\langle [\xi,x],Ty\rangle+ \langle [y,Tx],\xi\rangle-\langle [Tx,\xi],y\rangle+\langle [\xi,y],Tx\rangle&=4\langle x,y\rangle,\\
			\label{xixix}
			\langle [\xi,Tx],\xi\rangle&=0,
		\end{align} 
		\item[\riv]	for $x,y,z\in \zz^\perp$ with $x,y,z\perp\xi$,
		\begin{equation}\label{ky_equation}
			\la (\nabla_xT)y,z\ra+\la (\nabla_yT)x,z\ra=0.
		\end{equation} 
	\end{itemize}
\end{lem}

\begin{proof}
	Let $z_1,z_2\in \zz$ and $x\in \zz^\perp$. Firstly, we consider \eqref{ckyequationliealgebra} for $(x,z_1,z_2)$,
	$$\langle\left( \nabla_x T \right) z_1,z_2 \rangle +\langle \left( \nabla_{z_1} T \right) x,z_2 \rangle=2\langle x,z_1\rangle \theta(z_2)-\langle z_1,z_2\rangle \theta(x)-\langle x,z_2\rangle \theta(z_1).$$ By Lemma \ref{centro} $\theta(z_1)=\theta(z_2)=0$, then we have 
	\begin{align*}
	\langle \nabla_x T z_1,z_2 \rangle+\langle \nabla_x z_1,T z_2 \rangle +\langle  \nabla_{z_1} T  x,z_2 \rangle+\langle  \nabla_{z_1}  x,T z_2 \rangle&=-\langle z_1,z_2\rangle \theta(x)\\
	\frac{1}{2}\left(\langle [x,Tz_1],z_2\rangle+\langle [Tz_2,x],z_1\rangle-\langle [x,Tz_2],z_1\rangle \right) &=-\langle z_1,z_2\rangle \theta(x)
	\end{align*}
	\begin{equation}\label{aux1}
	\langle [x,Tz_1],z_2\rangle+2\langle [Tz_2,x],z_1\rangle =-2\langle z_1,z_2\rangle \theta(x).
	\end{equation}
	
	Secondly, we consider \eqref{ckyequationliealgebra} for $(z_2,z_1,x)$, then
	\begin{align*}
	\langle \left( \nabla_{z_2} T \right) z_1,x \rangle +\langle \left( \nabla_{z_1} T \right) z_2,x \rangle&=2\langle z_1,z_2\rangle \theta(x)\\
	\langle  \nabla_{z_2} T  z_1,x \rangle + \langle  \nabla_{z_2}  z_1,T x \rangle+\langle \nabla_{z_1} T  {z_2},x \rangle+\langle \nabla_{z_1}   z_2,Tx \rangle&=2\langle z_1,z_2\rangle \theta(x)\\
	\frac{1}{2}\left( \langle -[Tz_1,x],z_2\rangle-\langle [Tz_2,x],z_1\rangle \right) &=2\langle z_1,z_2\rangle \theta(x)
	\end{align*}
	\begin{equation}
	\label{aux2}
	\langle [x,Tz_1],z_2\rangle+\langle [x,Tz_2],z_1\rangle =4\langle z_1,z_2\rangle \theta(x).
	\end{equation}
	
%Note that \eqref{ckyequationliealgebra} for $(z_1,x,z_2)$ is the same that for $(x,z_1,z_2)$; and similarly with  $(z_1,z_2,x)$ and $(z_2,z_1,x)$. % We only need to see \eqref{ckyequationliealgebra} for $(z_2,x,z_1)$, but this result a combitacion of 
From \eqref{aux1} and \eqref{aux2} we get \eqref{lema21-1}.
In the same way we can obtain the other conditions.
\end{proof}

Note that,  \eqref{lema21-1} reduces to
\begin{align}
\label{lema22simple}
\langle [Tz_1,x],z_2\rangle=\langle [Tz_2,x],z_1\rangle=0,
\end{align}
if $z_1$ and $z_2$ are orthogonal. On the other hand, when $z_1=z_2=z$ we get
\begin{align}
\label{lema21simple}
\langle [Tz,x],z\rangle=-2\| z\|^2 \theta(x).
\end{align}

Finally, we see that  \eqref{ky_equation} applied to any fixed $x,y,z\in \zz^\perp$ with $x,y,z\perp\xi$ is exactly the KY condition and \eqref{xixix} is also described in Theorem \ref{ad1}. 
Since condition \eqref{ckyequationliealgebra} in $(u,v,w)$ is symmetric in $\{u,v\}$, then  the cases in  Lemma \ref{lema21} cover all possibilities for $u,v,w\in\ggo$ to satisfy the CKY condition \eqref{ckyequationliealgebra}, according to the decomposition $\ggo=\zz\oplus\zz^\perp$; therefore we have a converse result:

\begin{prop}\label{enough}
	Let $(\ggo,\la \cdot,\cdot \ra)$ be a metric Lie algebra with  $\on{dim}\zz \geq 2$ and  $T$ a skew-symmetric endomorphism  on $\ggo$. If there exists a non-vanishing vector $\xi\in\zz^\perp$ such that $T\xi=0$, %$T:\xi^\perp\to\xi^\perp$ is an isomorphism 
	and  \eqref{lema21-1}, \eqref{xyz1}, \eqref{xyz2}, \eqref{031}, \eqref{032}, \eqref{xixix} and \eqref{ky_equation} hold for $\theta(x)=\la \xi,x\ra$, then $T$ is a strict CKY tensor with associated $1$-form $\theta$ and dual vector $\xi$.	
\end{prop}

\begin{proof} 
We check the CKY condition \eqref{ckyequationliealgebra} for all possible choices of $x,y,z\in\ggo$. 
 Consider first the triples $(\xi,\xi,x)$ and $(\xi,x,\xi)$ with $x\in \zz^{\perp} \cap \xi^{\perp}$. Since \eqref{xixix} holds, the fact that $\nabla_x$ is skew-symmetric for all $x \in \ggo$ and that $T\xi=0$ implies that
	\begin{align*}
	0=&\la [\xi,Tx],\xi\ra=\la \nabla_{\xi}Tx-\nabla_{Tx}\xi,\xi\ra=\la \nabla_{\xi}Tx,\xi\ra+\la\nabla_{\xi}x,T\xi\ra=\la (\nabla_{\xi}T)x,\xi\ra=-\la (\nabla_{\xi}T)\xi,x\ra.
	\end{align*}
	Thus $\la (\nabla_{\xi}T)\xi,x\ra=0$ for all $x\in \zz^{\perp} \cap \xi^{\perp}$, and this is the CKY condition \eqref{ckyequationliealgebra} for $(\xi,\xi,x)$. Similarly, \eqref{xixix} implies that \eqref{ckyequationliealgebra} is satisfied for $(\xi,x,\xi)$.

	Consider now $(x,y,\xi)$ with $x,y\in  \zz^{\perp} \cap \xi^{\perp}$, since
	\eqref{031} holds, using that the torsion of the Levi-Civita connection vanishes, we get that
	\begin{align*}
	-2\langle x,y\rangle&=\langle \nabla_{\xi}Tx,y\rangle-\langle \nabla_{Tx}{\xi},y\rangle-\langle \nabla_ {Tx }y,\xi\rangle +\langle \nabla_{y}Tx,\xi\rangle+\langle \nabla_{y}\xi,Tx\rangle-\langle \nabla_{\xi}y,Tx\rangle +\\
	&+2\langle \nabla_{Ty}\xi,x\rangle -2\langle \nabla_{\xi}{Ty},x\rangle +2\langle \nabla_{Ty}x,\xi\rangle-2\langle \nabla_x {Ty},\xi\rangle\\
	&=2\langle \nabla_{\xi}Tx,y\rangle +2\langle \nabla_{\xi}x,{Ty}\rangle -2\langle \nabla_x {Ty},\xi\rangle=2\la \left( \nabla_{\xi}T\right) x,y\ra+2\langle \nabla_x \xi, {Ty}\rangle+2\la \nabla T\xi,y\ra\\
	&=2\la \left( \nabla_{\xi}T\right) x,y\ra+2\langle\left(  \nabla_xT\right)  \xi, {y}\rangle
	\end{align*}
	Thus $-\langle x,y\rangle =\la \left( \nabla_{\xi}T\right) x,y\ra+\langle\left(  \nabla_xT\right)  \xi, {y}\rangle$, which is the CKY condition \eqref{ckyequationliealgebra} for $(x,y,\xi)$. 
	%Then it is also satisfied \eqref{ckyequationliealgebra} por $(y,x,\xi)$. 
	In the same way, from \eqref{032} we get that 
	\[\la \left( \nabla_{x}T\right) y,\xi \ra+\langle\left(  \nabla_yT\right)  x, {\xi}\rangle=2 \la x,y\ra,\]
	hence the CKY conditions hold for $(x,\xi,y)$ with  $x,y\in  \zz^{\perp} \cap \xi^{\perp}$. 
	%Thus the CKY conditions hold also for $(\xi,x,y)$. \color{red} The CKY condition for $(y,\xi,x)$ can be deduced from \eqref{031} and \eqref{032}. \color{blue} Sacamos todas las simetricas mejor?

	Similarly, using that $\nabla_{z}x=\nabla_x z$  for all $z\in \zz$, we can see that \eqref{xyz1} and \eqref{xyz2} imply that $\la\left(  \nabla_{x}T\right) y,z\ra+\la\left(  \nabla_{y}T\right) x,z\ra=0$ and $\la\left(  \nabla_{x}T\right) z,y\ra+\la\left(  \nabla_{z}T\right) x,y\ra=0$ for all $x,y\in \zz^{\perp}$. 
	Therefore, the CKY condition is satisfied for  all $x,y\in \zz^{\perp}$ and $z\in \zz$.

	Now for any $x,y,z\in \zz^{\perp} \cap \xi^{\perp}$ it follows from \eqref{ky_equation} that the CKY condition \eqref{ckyequationliealgebra} is verified.
	Next, as \eqref{lema21-1} holds, using that $\la \nabla_{z_1}z_2,x\ra=0$, and that $\nabla_{z_1}x=\nabla_x {z_1}$ for all $z_1,z_2\in \zz$ and $x\in \zz^{\perp}$, we can see that the CKY condition  is satisfied in this case. 
	
	Finally, the CKY condition \eqref{ckyequationliealgebra} for $z_1,z_2,z_3\in \zz$ is trivially verified.
	Since  \eqref{lema21-1}, \eqref{xyz1}, \eqref{xyz2}, \eqref{031}, \eqref{032},  \eqref{xixix} and \eqref{ky_equation} cover all the possibilities for $x,y,z\in \ggo$, we can conclude that $T$ is a CKY tensor.
\end{proof}

\section{$5$-dimensional Lie algebras with $\dim\zz>1$ admitting CKY $2$-forms }\label{LieAlgebras}

In this section we consider a $5$-dimensional non-abelian metric Lie algebra $(\ggo,\la \cdot,\cdot \ra)$  with $\dim\zz>1$, thus the center has dimension either $2$ or $3$. The next lemma shows that when $\dim \zz=3$, there are no strict CKY tensors.

\begin{lem}\label{centro3}
A $5$-dimensional metric Lie algebra $(\ggo,\la \cdot,\cdot\ra)$   with  $\dim \zz=3$ does not admit any strict CKY tensor.
\end{lem}
\begin{proof}
 It follows from Lemma \ref{centro} that $\zz\perp \xi$. 
	Let $z_1,z_2,z_3$ be an orthonormal basis of $\zz$ and let $x\in \ggo$ be a unitary vector such that $x\perp \zz$ and $\la x,\xi\ra=0$. Then the only non-zero Lie bracket is given by
	$$[x,\xi]=a_1x+a_2\xi+a_3z_1+a_4z_2+a_5z_3\neq 0,$$
for some $a_1,a_2,a_3,a_4,a_5\in\mathbb R$. Suppose $T$ is a strict CKY tensor on $\ggo$ given by the following matrix in the orthonormal basis $\{\xi,z_1,z_2,z_3,x\}$
	\begin{align}\label{T}
	T=\left[ \begin{matrix}0&0&0&0&0\\
	0&0&-a&-b&-c\\0&a&0&-d&-e\\0&b&d&0&-f\\0&c&e&f&0\end{matrix}\right] 
	\end{align}
	Using \eqref{lema21simple}, that is, $\la [Tz_i,\xi],z_i\ra=-2$ for $i=1,2$,  we get $ca_3=ea_4=-2$ and therefore, $c,e,a_3,a_4\neq 0$. 
	On the other hand, from  \eqref{lema22simple} for $(z_1,z_2,\xi)$ we obtain $ca_4=0$, which is a contradiction.
	\end{proof}

\

Therefore, we will focus on the case $\dim\zz=2$. Let $(\ggo, \la \cdot , \cdot \ra)$ be a $5$-dimensional metric Lie algebra with center of dimension $2$, and $T$ a strict CKY tensor with associated $1$-form $\theta$ and dual unitary vector $\x$.
We can decompose $\ggo=\zz \oplus \zz^\perp$, and let $\{z_1,z_2\}$ be an orthonormal basis of the center.
It follows from Lemma \ref{centro} that $\xi$ is orthogonal to $\zz$, we can therefore complete to an orthonormal basis  $\{\xi, x,y\}$ of $\zz^\perp$. 
Then $T$ can be written  as in \eqref{T} in the orthonormal basis $\{\xi,z_1,z_2,x,y\}$ of $\ggo$, 
for some $a,b,c,d,e,f\in\mathbb R$, such that $T | _{\xi^ {\perp}}$ is an isomorphism, that is, $af+cd-be\neq 0$.
The Lie brackets on $\ggo$ can be described as follows
\begin{align}\label{corchetes}
[\xi,x]&=a_1z_1+a_2z_2+a_3x+a_4y+a_5\xi,\nonumber\\ 
[\xi,y]&=b_1z_1+b_2z_2+b_3x+b_4y+b_5\xi,\\
[x,y]&=c_1z_1+c_2z_2+c_3x+c_4y+c_5\xi,\nonumber
\end{align}
for some real numbers $a_i,b_i,c_i$, and this Lie algebra is denoted by $\ggo_{a_i,b_i,c_i}$.

%{\color{blue}
Next, we reduce the possibilities for the structure constants $a_i,b_i,c_i$ and the parameters $a,b,c,d,e,f$ of $T$, in order to satisfy the CKY condition \eqref{ckyequationliealgebra}. 
Note first that by doing a rotation in the center, one can always choose $a_2=0$ in \eqref{corchetes}. In the next lemma a further reduction is shown.

\begin{lem}\label{LBconditions}
	Let $(\ggo,\ip)$ be a $5$-dimensional metric Lie algebra with center of dimension $2$. Then $(\ggo,\ip)$ admits a strict CKY  tensor if and only if there exists an orthonormal basis $\{\xi,z_1,z_2,x,y\}$ %and Lie brackets as in \eqref{corchetes} 
	such that:
	\begin{align}\label{super-reduced-Lie_brackets}
	[\xi,x]&=rz_1+a_4y,&
	[\xi,y]&=rz_2-a_4x,&
	[x,y]&=s\xi,
\end{align}
for some $r>0$, $s>0$, and  $a_4=\dfrac{s^2-r^2}{s}$. 
	Moreover, the CKY	$2$-form is given by  
\begin{equation}\label{Tisometrico}
		\omega=\frac{2(s^2+2r^2)}{r^2s}z^1\wedge z^2 + \frac{2}{r}(z^1\wedge x^* +z^2\wedge y^*) + \frac{4}{s}x^*\wedge y^*,
	\end{equation}
where $\{\theta,z^1,z^2,x^*,y^*\}$ is the metric dual basis of $\{\xi,z_1,z_2,x,y\}$. 
\end{lem}

\begin{proof}
	Suppose $(\ggo,\ip)$ admits a strict CKY tensor $T$.
	We can decompose $\ggo=\zz\oplus\zz^{\perp}$ with $\xi \in \zz^{\perp}$, and $\xi^{\perp}=\zz\oplus(\zz^\perp\cap\xi^\perp)$.
	The operators  $T | _{\xi^ {\perp}}$ and $\ad_\xi| _{\xi^ {\perp}}$ can be written as block matrices, adapted to this decomposition, as follows
	$$T | _{\xi^ {\perp}}=\begin{pmatrix} M&-Q^t\\ Q&N \end{pmatrix}, 
	\; 
	\ad_{\xi} | _{\xi^ {\perp}}=\begin{pmatrix} 0&A\\ 0&B \end{pmatrix}$$
	where $M,Q,N$ are $2\times2$ matrices with $M$ and $N$  skew-symmetric. 

	We choose an orthonormal basis $\{\xi,z_1,z_2,x,y\}$ of $\ggo$ adapted to this decomposition such that the Lie brackets are as in \eqref{corchetes}, and $T$ is as in \eqref{T}, that is, 
	\begin{align*}
	M&=\begin{pmatrix} 0&-a\\ a&0 \end{pmatrix}, &
	N&=\begin{pmatrix} 0&-f\\ f&0 \end{pmatrix}, &
	Q&=\begin{pmatrix} b&d\\ c&e \end{pmatrix}, &
	A&=\begin{pmatrix} a_1&b_1\\ 0&b_2 \end{pmatrix}, &
	B&=\begin{pmatrix} a_3&b_3\\ a_4&b_4 \end{pmatrix}.\end{align*}
	Recall that the operator $T | _{\xi^ {\perp}}  : {\xi}^ {\perp} \to {\xi}^ {\perp}$ is a linear isomorphism according to Theorem \ref{ad1}.
	We analyze now the operator 
	$\ad_{\xi}T$ and the operator $T\ad_{\xi}$. Notice that Theorem \ref{ad1} implies that $\text{Im}(\ad_{\xi})\subset\xi^\perp$, that is, $a_5=b_5=0$.
	Replacing, $\xi,z_1,z_2$ in \eqref{lema21-1}, we obtain
		$$AQ=2\id,$$
		where $AQ$ is the projection of the operator $\ad_{\xi}T|_{\zz}$  to $\zz$, and $\id$ denotes the identity matrix. In particular, $a_1b_2\neq0$, $c=0$, and  $be\neq0$.
	It is easy to see that $[x,Tz_2]=e[x,y]$ and $[y,Tz_1]=-b[x,y]$. Now from \eqref{lema21-1} again we have that $[x,y]\in\zz^\perp$, and therefore $c_1=c_2=0$.

	On the other hand, from \eqref{xyz2} we have that $\langle [x,y],Tz\rangle=0$, or equivalently  $Q^t[x,y]=0$. Since $Q$ is invertible, we obtain that $c_3=c_4=0$, and therefore $[x,y]|_{\xi^\perp}=0$, that is $[x,y]=c_5\xi$ with $c_5\neq0$ according to Lemma \ref{xiperpgg}. 
	
	Notice that $[x,Ty]=[y,Tx]=0$ and $[x,Tx]=[y,Ty]=fc_5\xi$. Replacing this into \eqref{031} we obtain information about the projection of the operator $\ad_{\xi}T|_{\zz^\perp\cap\xi^\perp}$  to $\zz^\perp\cap\xi^\perp$, that is, 
	$$BN+(NB)^t-2(BN)^t=(fc_5-4)\id,$$ which is equivalent to  $b_4=a_3=0$, $b_3=-a_4$ and $c_5f=4$. In particular, we have that $NB=BN=-fa_4\id$.
	
	Next from \eqref{xyz1}  and \eqref{xyz2} we obtain that the operator $\ad_{\xi}T|_{\zz}$ projected to $\zz^\perp\cap\xi^\perp$ has to satisfy $$(AN)^t-2BQ+2W=0, \; -3AN+2(MA-Q^tB)=0, \; \text{with} \; W=\begin{pmatrix}
		0&-c_5\\c_5&0
	\end{pmatrix}Q.$$
	The first equality leads to,
    $A^tA=-c_5(a_4-c_5)\id$,
    thus $b_1=0$, $b_2=\pm a_1$ and $a_4=\frac{c_5^2-a_1^2}{c_5}$.
    Moreover, both cases $b_2=\pm a_1$ reduce to each other by changing the
    	signs of $z_2$ and $b_2$. 
    Note also that we can assume $c_5>0$ by changing the signs of $\xi,z_1,z_2$, and $a_1>0$ by changing the signs of $z_1,z_2$.
    Therefore,
    the structure constants \eqref{corchetes} reduce to \eqref{super-reduced-Lie_brackets} where $s=c_5$ and $r=a_1$. 
    Finally, the second equality shows that 
    $M=\frac{2}{r^2}B+ \frac{3}{2}N$, and then  
    the $2$-form associated to $T$ is given by \eqref{Tisometrico}.

	\medskip

	Now, assume there exists an orthonormal basis $\{\xi,z_1,z_2,x,y\}$ of $(\ggo,\ip)$ with Lie brackets as in \eqref{super-reduced-Lie_brackets}. It can be shown that the tensor associated to the CKY $2$-form $\omega$ given by \eqref{Tisometrico} satisfies the hypothesis of Proposition \ref{enough}. Therefore, this it is a strict CKY $2$-form for the metric Lie algebra $(\ggo,\ip)$.
\end{proof}

We denote the metric Lie algebras of Lemma \ref{LBconditions} with Lie bracket as in \eqref{super-reduced-Lie_brackets} by $(\ggo_{r,s},\ip)$ with $r,s>0$.
\begin{rem}
	As a consequence of Lemma \ref{LBconditions}, we have that 
    a $5$-dimensional metric Lie algebra with center of dimension $2$ admitting a strict CKY tensor has to be unimodular. Also note that a given metric Lie algebra $(\ggo_{r,s},\ip)$ admits only one strict CKY tensor, up to scaling.
\end{rem}

\begin{lem}\label{isometric}
Let $(\ggo_{r,s} , \la \cdot , \cdot \ra)$ be as above, then 
	$(\ggo_{r,s} , \la \cdot , \cdot \ra)$ are pairwise non-isometrically isomorphic for $r,s>0$.
\end{lem}

\begin{proof}

 Let $\varphi: (\ggo_{r,s} , \la \cdot , \cdot \ra)\to(\ggo_{r',s'} , \la \cdot , \cdot \ra)$ be an arbitrary isometric isomorphism with $r,s,r',s'>0$.
  Since $\varphi$ is a Lie isomorphism, $\varphi$ preserves the center, i.e. $\varphi(\zz(\ggo_{r,s}))=\zz(\ggo_{r',s'})$. In addition,  using that $\varphi([x,y])=[\varphi(x),\varphi(y)]$ and the fact that $\varphi$ is an isometry, then
 $\varphi$ can be written in the orthonormal basis $\{\xi,z_1,z_2,x,y\}$ as $\varphi=\begin{pmatrix}
 	a&0&0&0&0\\
 	0&u&-bv&0&0\\
 	0&v&bu&0&0\\
 	0&0&0&g&-ch\\
 	0&0&0&h&cg
 \end{pmatrix}$,
for some $a,b,c,g,h,u,v\in\mathbb R$ with $a^2=b^2=c^2=1$, and $u^2+v^2=g^2+h^2=1$. From
 $\varphi([x,y])=[\varphi(x),\varphi(y)]$  again  one can show that $a=c$ and $s=s'$, and from $\varphi([\xi,x])=[\varphi(\xi),\varphi(x)]$ we obtain that $r=r'$. Moreover, using that $\varphi([\xi,y])=[\varphi(\xi),\varphi(y)]$, and once again $\varphi([\xi,x])=[\varphi(\xi),\varphi(x)]$ we have that $a=b$ and
 $\varphi$ has the following form
	\[ \varphi=
	\begin{pmatrix}
		a&0&0&0&0\\
		0&a\cos t&-\sin t&0&0\\
		0&a\sin t&\cos t&0&0\\
		0&0&0&\cos t&-a\sin t\\
		0&0&0&\sin t&a\cos t
	\end{pmatrix}\]
with $a^2=1$.
\end{proof}

\subsection{Classification}

Now, in order to determine exactly which metric Lie algebra we have obtained, we consider two cases:

\subsubsection{Case $r=s$:} 

Consider the metric Lie algebra $(\ggo_{r,s},\ip)$ with $s=r$, then $a_4=0$. By making a change of basis $\{E=r\xi, X=x, Y=y, Z_1=-r^2z_1, Z_2=-r^2z_2\}$ we have the following Lie brackets
\begin{align}\label{LieBracket-L59}
	[X,E]&=Z_1&
	[Y,E]&=Z_2 & 
	[X,Y]&=E.
\end{align}
This Lie algebra $\ggo_{r,r}$ is isomorphic to the $3$-step nilpotent Lie algebra $L_{5,9}$ (see \cite{deGraaf}). 
Now we can rewrite this metric $\ip$ in the canonical basis $\{E,Z_1,Z_2,X,Y\}$, and it has the following matrix representation denoted by $\ip_{r}$:
\[\ip_{r}=\begin{pmatrix}
	r^2&0&0&0&0\\
	0&r^4&0&0&0\\
	0&0&r^4&0&0\\
	0&0&0&1&0\\
	0&0&0&0&1
\end{pmatrix}.\]
Therefore, $(\ggo_{r,r},\ip)$ is isometrically isomorphic to $(L_{5,9},\ip_r)$. Moreover, it can be shown by a direct computation that $(L_{5,9},\ip_r)$ is isometrically isomorphic to $(L_{5,9},\ip_{r'})$ if and only if $r=r'$ (see also Lemma \ref{isometric}). %  \Mar{Incluimos la cuenta?} 
Therefore, there is a one-parameter family of metrics on $L_{5,9}$ admitting a CKY tensor.
For a complete description of the isometry classes of all metrics on this Lie algebra we refer to \cite{FN}.

\subsubsection{Case $r\neq s$:}
Consider the metric Lie algebras $(\ggo_{r,s},\ip)$ with $r\neq s$, that is, $a_4\neq0$. By making  a change of basis $\{ \xi, z_1,z_2, \overline{X}=rz_1+a_4y, \overline{Y}=rz_2-a_4x\}$ we get the following Lie brackets
\begin{align*}
	[\xi,\overline{X}]&=a_4\overline{Y}&
	[\overline{Y},\xi]&=a_4\overline{X} & 
	[\overline{X},\overline{Y}]&=a_4^2s\xi.
\end{align*}
Suppose now that $a_4>0$, that is $s>r$,
then we can change the basis as follows: 
\begin{align*}
	X:&=\frac{1}{\sqrt{a_4^3s}}\overline{X}, & Y:&=-\frac{1}{\sqrt{a_4^3s}}\overline{Y}, & E:&=-\frac{1}{a_4}\xi, & Z_1:&=z_1,& Z_2:&=z_2.
\end{align*} 
It is easy to see that $\ggo$ is isomorphic to $\RR^2 \times \mathfrak{su}(2)$ with Lie brackets
\begin{align*}
	[E,X]&=Y,&
	[X,Y]&=E,&
	[Y,E]&=X.
\end{align*}
In this basis $\{E,Z_1,Z_2,X,Y\}$ the metric $\ip$ has the following form denoted by $\ip_{r,s}$:

$${\ip_{r,s}}=
\begin{pmatrix}
	\frac{1}{a_4^2}&0&0&0&0\\
	0&1&0&\frac{r}{\sqrt{a_4^3s}}&0\\
	0&0&1&0&-\frac{r}{\sqrt{a_4^3s}}\\
	0&\frac{r}{\sqrt{a_4^3s}}&0&\frac{r^2+a_4^2}{a_4^3s}&0\\
	0&0&-\frac{r}{\sqrt{a_4^3s}}&0&\frac{r^2+a_4^2}{a_4^3s}
\end{pmatrix},$$
with $a_4=\frac{s^2-r^2}{s}$.  Moreover, it follows from Lemma \ref{isometric} that these metrics $\ip_{r,s}$ are pairwise non-isometric. Indeed, if $(\RR^2 \times \mathfrak{su}(2),\ip_{r,s})\simeq(\RR^2 \times \mathfrak{su}(2),\ip_{r',s'})$ then $(\ggo_{r,s},\ip)\simeq(\ggo_{r',s'},\ip)$ and therefore $r=r'$ and $s=s'$.

\medskip

Similarly, when $a_4<0$, that is $s<r$, 
we can set the basis $\{E,Z_1,Z_2,X,Y\}$ with
\begin{align*}
	X:&=\frac{1}{\sqrt{|a_4|^3s}}\,\overline{X}, & Y:&=\frac{1}{\sqrt{|a_4|^3s}}\,\overline{Y}, & E&=\frac{1}{a_4}\xi, & Z_1&:=z_1,& Z_2&:=z_2.
\end{align*}
It can be seen that $\ggo$ is isomorphic to $\RR^2 \times\mathfrak{sl}(2,\RR)$ with Lie brackets
\begin{align*}
	[E,X]&=Y,&
	[X,Y]&=-E,&
	[Y,E]&=X.
\end{align*}
The metric $\ip$ in this basis $\{E,Z_1,Z_2,X,Y\}$ has the following form
$${\ip_{r,s}}=
\begin{pmatrix}
	\frac{1}{a_4^2}&0&0&0&0\\
	0&1&0&\frac{r}{\sqrt{|a_4|^3s}}&0\\
	0&0&1&0&\frac{r}{\sqrt{|a_4|^3s}}\\
	0&\frac{r}{\sqrt{|a_4|^3s}}&0&\frac{r^2+a_4^2}{|a_4|^3 s}&0\\
	0&0&\frac{r}{\sqrt{|a_4|^3s}}&0&\frac{r^2+a_4^2}{|a_4|^3s}
\end{pmatrix},$$
with $a_4=\frac{s^2-r^2}{s}$. Similar to the case above, it follows from Lemma \ref{isometric} that these metrics $\ip_{r,s}$ are pairwise non-isometric.

\begin{rem}
	Note that the factors in each Lie algebra $\RR^2 \times\mathfrak{su}(2)$ and $\RR^2 \times\mathfrak{sl}(2,\RR)$ are not orthogonal with respect to  $\ip_{r,s}$.
\end{rem}

We summarize in the following result the classification of all  $5$-dimensional metric Lie algebras with center of dimension greater than one admitting strict CKY tensors, up to isometric isomorphism and scaling.

\begin{thm}\label{metricLA}
	Let $(\ggo,\ip)$ be a $5$-dimensional metric Lie algebra with $\dim\zz>1$. If $(\ggo,\ip)$ admits a strict CKY tensor, then $(\ggo,\ip)$ is isometrically isomorphic to one and only one of the following:
	$(L_{5,9},\ip_r)$, $(\RR^2 \times\mathfrak{su}(2),\ip_{r,s})$, $(\RR^2 \times\mathfrak{sl}(2,\RR),\ip_{r,s})$
	for $r,s>0$, $r\neq s$.
	Moreover, the CKY tensor is uniquely determined by the metric, up to scaling and it %can be represented by the matrix
	is given in Table \ref{TensoresBaseCanonica}.
\end{thm}

\begin{table}[h!]
	\begin{tabular}{ c c c }
		\hline \hline
		{\scriptsize Metric Lie algebra }  & {\scriptsize strict CKY $2$-form} & {\scriptsize parameters}\\ \hline \hline 
%		& & \\
		\begin{tiny}
%			$(L_{5,9},\ip_r)$
			$\begin{matrix} (L_{5,9},\ip_r)\\ 
				&\\
				[X,E]=Z_1\\
				[Y,E]=Z_2\\
				[X,Y]=E\end{matrix}$
		\end{tiny} 
& 
		
%		\begin{tiny}
			$6r^3z^1\wedge z^2 - 2r(z^1\wedge x^* + z^2\wedge y^*) + \frac{4}{r}x^*\wedge y^*$
%			$\begin{pmatrix} 
%				0 & 0 & 0 & 0 & 0 \\
%				0&0& -\frac{6}{r}&\frac{2}{r^3}&0 \\
%				0&\frac{6}{r}&0&0&\frac{2}{r^3} \\
%				0&-2r& 0& 0&-\frac{4}{r} \\
%				0&0& -2r& \frac{4}{r}&0 \\
%			\end{pmatrix}$
%		\end{tiny} 
& $r>0$
		\\ \hline
%		 \\
		\begin{tiny}			
			$\begin{matrix}
				(\RR^2 \times\mathfrak{su}(2),\ip_{r,s})\\
				&\\
				[E,X]=Y\\
				[X,Y]=E\\ 
				[Y,E]=X\end{matrix}$
		\end{tiny}
		&
		
%		\begin{tiny}
			$\frac{2(s^2+2r^2)}{r^2s} z^1\wedge z^2 -  \frac{6r}{(s^2-r^2)^{3/2}}(z^1\wedge y^* + z^2\wedge x^*)  - \frac{2(s^2+2r^2)}{s(s^2-r^2)^3}x^*\wedge y^*$			
%			$\begin{pmatrix} 
%				0 & 0 & 0 & 0 & 0 \\
%				0&0& -\frac{2(s^4+2s^2r^2-2r^4)}{r^2s(s^2-r^2)}&0&\frac{2r(2s^2-r^2)}{(s^2-r^2)^{5/2}}\\
%				0&\frac{2(s^4-2s^2r^2+2r^4)}{r^2s(s^2-r^2)}&0&\frac{2r(2s^2-r^2)}{(s^2-r^2)^{5/2}}&0 \\
%				0&0& \frac{2\sqrt{s^2-r^2}}{r}& 0&\frac{2(s^2-2r^2)}{s(s^2-r^2)}\\
%				0&\frac{2\sqrt{s^2-r^2}}{r}& 0& -\frac{2(s^2-2r^2)}{s(s^2-r^2)}&0 \\
%			\end{pmatrix}$
%		\end{tiny}
		&  $s>r>0$
		
	\\ \hline
	 
		\begin{tiny}	
			$\begin{matrix}
				(\RR^2 \times\mathfrak{sl}(2,\RR),\ip_{r,s})\\
				&\\
				[E,X]=Y\\
				[X,Y]=-E\\ 
				[Y,E]=X\end{matrix}$
		\end{tiny}
		
		& 
%		\begin{tiny}
			 $\frac{2(s^2+2r^2)}{r^2s} z^1\wedge z^2 +  \frac{6r}{(r^2-s^2)^{3/2}}(z^1\wedge y^* - z^2\wedge x^*)  - \frac{2(s^2+2r^2)}{s(s^2-r^2)^3}x^*\wedge y^*$ 
						
%			$\begin{pmatrix} 
%				0 & 0 & 0 & 0 & 0 \\
%				0&0& -\frac{2(s^4-2r^4)}{r^2s(s^2-r^2)}&0&-\frac{2(2s^4-4s^2r^2+r^4)}{r(r^2-s^2)^{5/2}}\\
%				0&\frac{2(s^4-2r^4)}{r^2s(s^2-r^2)}&0&-\frac{2(2s^4-4s^2r^2+r^4)}{r(r^2-s^2)^{5/2}}&0 \\
%				0&0& \frac{2\sqrt{r^2-s^2}}{r}& 0&\frac{2(3s^2-2r^2)}{s(s^2-r^2)}\\
%				0&\frac{2\sqrt{r^2-s^2}}{r}& 0& -\frac{2(3s^2-2r^2)}{s(s^2-r^2)}&0 \\
%			\end{pmatrix}$
%		\end{tiny}
		&
		$r>s>0$ \\
	 \hline
	\end{tabular}
	\caption{Notation: $\{e,z^1,z^2,x^*,y^*\}$ is the metric dual basis of $\{E,Z_1,Z_2,X,Y\}$.}
	\label{TensoresBaseCanonica}
\end{table}

\begin{rem}
	These $5$-dimensional metric Lie algebras exhibited in Theorem \ref{metricLA} represent the first explicit examples of Lie algebras carrying strict CKY $2$-forms and not admitting any Sasakian structure.
\end{rem}

\

\subsection{Compact examples}

The CKY tensors on the metric Lie algebras determined in Table \ref{TensoresBaseCanonica} can be used to produce examples of compact manifolds admitting CKY tensors.
Recall that any left invariant Riemannian metric on a Lie group $G$ defines a Riemannian metric on any compact quotient $\Gamma\backslash G$, such that the natural projection $p : G \to \Gamma\backslash  G$ is a locally isometric covering. 
Therefore, any CKY tensor on a metric Lie algebra $(\ggo,\ip)$
induces a CKY tensor on any compact quotient $\Gamma\backslash G$, where $\Gamma$ is a discrete subgroup of $G$.

Consider  
the metric Lie algebra $(L_{5,9},  \ip_r)$ with the CKY tensor given in Table \ref{TensoresBaseCanonica}.
This is an example of a $3$-step nilpotent non-Sasakian Lie algebra admitting a CKY tensor. The associated simply connected Lie group admits lattices (\cite{Malcev}). Moreover, any lattice in the associated simply connected Lie group determines a nilmanifold admitting a CKY tensor induced by $T$.
In \cite{CMdNMY} it is shown that the only nilmanifolds admitting a Sasakian structure (not necessarly invariant) are quotients of the Heisenberg Lie group. Therefore, any nilmanifold obtained as a quotient of the simply connected Lie group with Lie algebra $L_{5,9}$ does not admit any Sasakian structure.

For the second metric Lie algebra $(\RR^2 \times\mathfrak{su}(2),\ip_{r,s})$, any lattice in the associated simply connected Lie group $\RR^2 \times SU(2)$ determines a compact manifold admitting a CKY tensor induced by $T$. In particular, the product $\mathcal S^3\times \mathcal T^2$ of the $3$-dimensional sphere with the $2$-dimensional torus equipped with the metric induced by $\ip_{r,s}$  admits a CKY tensor; in contrast with the nilmanifold in the first example, this Lie group does admit Sasakian structures (which are non-invariant), according to \cite{BTF}.

Finally, the metric Lie algebra $(\RR^2 \times\mathfrak{sl}(2,\RR),\ip_{r,s})$ is the third example of a non-Sasakian Lie algebra admitting a CKY tensor. 
The Lie group $SL(2,\RR)\times\RR^2$ admits lattices (see \cite{Mos75}).
Moreover, any lattice determines a compact manifold admitting a CKY tensor induced by $T$. 
We do not known in this case if these compact manifolds admit non-invariant Sasakian structures.

%\section{Further properties}

\section{The vector space of CKY 2-forms}

%In this section we analyze other characteristics of the metric Lie algebras classified in Section 3 and Section 5.  
%We first study the space of solution of CKY $2$-forms. Secondly, we describe the holonomy group of the Riemannian simply connected Lie group associated to each metric Lie algebra in Table \ref{tabla1} and each metric Lie algebra in Table \ref{TensoresBaseCanonica}. % and study its action.

%\subsection{The vector space of CKY 2-forms}

The space of CKY 2-forms on a given metric Lie algebra $(\ggo,\ip) $ is a vector space (see \cite{Stepanov}, \cite{Semmelmann}), we denote it by $\mathcal{CKY}^2(\ggo,\ip)$. It contains two distinguished vector subspaces: the vector space of KY $2$-forms   and the vector space of $*$-KY $2$-forms  denoted by $\mathcal{KY}^2(\ggo,\ip)$ and $\mathcal{*KY}^2(\ggo,\ip)$ respectively. 
It is known that the sum of any $*$-KY $2$-form and any KY $2$-form is a CKY $2$-form, but the converse does not hold in general, see \cite{Kashiwada, Tachibana}. In this section we investigate these subspaces on the metric Lie algebras in Table \ref{tabla1}, and then any metric Lie algebra in Table \ref{TensoresBaseCanonica}. %$(\ggo_{r,s},\ip)$ for $r,s>0$.

\smallskip

Regarding the metric Lie algebras from Section 3, we can characterize the CKY $2$-forms classified in Theorem \ref{classificacion_1} further. Indeed, it is clear that those CKY $2$-forms are not Killing-Yano; however, in the next proposition, we show that they are always closed.

\begin{prop}\label{closed}
	Let $(\ggo,\ip)$ be a $5$-dimensional metric Lie algebra,
	%with $\dim\zz=1$. 
	%If $(\ggo,\ip)$ admits a strict CKY tensor $T$ with associated vector $\xi\in\zz$, then the $2$-form associated to $T$ is closed. 
then any strict CKY $2$-form, such that its co-differential lies in the dual of the center, is a $*$-KY $2$-form.
\end{prop} 

\begin{proof}
		Let $\omega$ be a strict CKY $2$-form with associated tensor $T$ and central vector $\xi$. 
		It follows from Theorem \ref{teorema dotti-andrada} and the discussion there, that $(\ggo,\ip)$ is as central extension of a $4$-dimensional Lie algebra equipped with a parallel skew-symmetric tensor. 
		Moreover, using the fact that $T\xi=0$ we obtain $\mathrm{d}\omega(x,y,z)=\mathrm{d}^\mathfrak h\omega(x,y,z)$ for any $x,y,z\in \mathfrak h$, where $\mathrm{d}^\mathfrak h$ denotes the derivative with respect to $(\mathfrak{h},[\,,\,]_{\mathfrak{h}},\langle \cdot,\cdot \rangle|_{\mathfrak{h}\times\mathfrak{h}})$. 
		
		Since $T|_\mathfrak{h}$ is a parallel tensor on  $(\mathfrak{h},\langle \cdot,\cdot \rangle|_{\mathfrak{h}\times\mathfrak{h}})$, we have that $\mathrm{d}^\mathfrak h\omega=0$ (see \cite[Corollary 2.2]{AD}).
		In addition, $\mathrm{d}\omega(x,y,\xi)=-\omega([x,y],\xi)=\langle [x,y],T\xi\rangle=0$. Therefore, $\omega$ is a closed, that is, $\omega$ is a $*$-KY $2$-form. 
\end{proof}

Let $(\ggo, \ip)$ be a $5$-dimensional metric Lie algebra admitting a strict CKY  $2$-form such that its co-differential lies in the dual of the center, or equivalently a metric Lie algebra classified in Theorem \ref{classificacion_1}; we notice that there is only one strict CKY  $2$-form, up to scaling. Indeed, any associated tensor in the third column of Table \ref{tabla1} is completely determined by the corresponding metric Lie algebra, and any non-zero multiple of $T$ is again a strict CKY tensor according to Remark \ref{rescaling}. 

In the next result we show that each metric Lie algebra in Table \ref{tabla1} does not admit any other CKY $2$-forms, including KY $2$-forms and strict CKY $2$-forms whose co-differential is not necessary in the dual of the center.

\begin{thm}\label{CKYspace_extensiones}
	Let $(\ggo,\ip)$ be a $5$-dimensional metric Lie algebra admitting a strict CKY \linebreak $2$-form, such that its co-differential lies in the dual of the center, then 
	%the space of CKY 2-forms $\mathcal{CKY}^2(\ggo,\ip)$ is $1$-dimensional.
	$\mathcal{KY}^2(\ggo,\ip)=0$ and $\mathcal{*KY}^2(\ggo,\ip)=\mathcal{CKY}^2(\ggo,\ip)$ is $1$-dimensional.  In particular, $(\ggo,\ip)$ does not admit non-zero parallel tensors.
\end{thm}

\begin{proof}
	It follows from Section 3 that any $5$-dimensional metric Lie algebra admitting a CKY $2$-form such that its co-differential lies in the dual of the center is listed in Table \ref{tabla1}. Thus, let us fix one of the Lie algebras $(\ggo,\ip)$ in Table \ref{tabla1}, and we proceed case by case. 
	By a direct calculation assisted by the software system Sage \cite{SAGE} we see that an arbitrary $2$-form on $(\ggo,\ip)$ is CKY, if and only if, its associated tensor is the one listed in Table \ref{tabla1}, up to scaling.  
	
	In particular, $(\ggo,\ip)$ does not admit neither non-trivial KY $2$-forms nor strict CKY $2$-forms such that their co-differential is not in the dual of the center (we are happy to share the Sage's implementation with anybody who might be interested). Therefore, $\mathcal{KY}^2(\ggo,\ip)=0$, $\mathcal{CKY}^2(\ggo,\ip)$ is $1$-dimensional, and using Proposition \ref{closed} it follows that $\mathcal{*KY}^2(\ggo,\ip)=\mathcal{CKY}^2(\ggo,\ip)$.
\end{proof}
	
\begin{rem}
	It follows from Theorem \ref{CKYspace_extensiones} that any $5$-dimensional metric Lie algebra $(\ggo,\ip)$
	admitting a strict CKY $2$-form, such that its co-differential lies in the dual of the center,  does not admit any other  strict  CKY $2$-form $\omega$ such that the associated vector $\xi$ is not in the center. 
	However, it is still an open problem to determine the existence of CKY $2$-forms on an arbitrary metric Lie algebra with $1$-dimensional center, which is not necessarily generated by $\xi$.
	In particular, it would be interesting to study other metrics on those  $5$-dimensional Lie algebras in Table \ref{NotacionNueva} and determine if they admit a CKY $2$-form with $\xi\notin \zz$. In that case, the metric will not be isometric to the ones described in Table \ref{NotacionNueva}.
\end{rem}

%\Mar{Sera que vale en general que si una $(\ggo,\ip)$ admite un CKY con $\xi\in\zz$ entonces no admite  otros con $\xi\notin\zz$??? Habria que estudiar bien la extension central...}

\

We consider now a $5$-dimensional metric Lie algebra with $\dim\zz>1$ admitting a strict CKY tensor, that is,
a metric Lie algebra  classified in Theorem \ref{metricLA}, see Table \ref{TensoresBaseCanonica}. Since they are isometrically isomorphic to $(\ggo_{r,s},\ip)$ for some $r,s>0$ we can work with the latter for simplicity.
Given $(\ggo_{r,s},\ip)$, a strict CKY tensor $T$ is completely determined by $r,s$ according to  \eqref{Tisometrico}. 
Moreover, it follows from Remark \ref{rescaling} that any multiple of that $T$ is a strict CKY tensor as well. 
Thus, we can conclude that there is only one, up to scaling, strict CKY 2-form for each $(\ggo_{r,s},\ip)$, in particular, $1\leq \dim \,\mathcal{CKY}(\ggo_{r,s},\ip) \leq$ ${7}\choose{3}$, according to \cite[Theorem 5.2]{Semmelmann}. We show in the following proposition that  these strict CKY 2-forms are never $*$-KY (compare with Proposition \ref{closed}).

\begin{prop}\label{never_closed}
	Let $(\ggo,\ip)$ be a $5$-dimensional metric Lie algebra with $\dim\zz>1$, then any strict CKY  $2$-form $\omega$ is never closed.
\end{prop}
\begin{proof}
	By Lemma \ref{LBconditions} and Lemma \ref{isometric}, a metric Lie algebra $(\ggo,\ip)$ admitting a strict CKY  $2$-form $\omega$ is isomorphically isometric to $(\ggo_{r,s},\ip)$ for some $r,s>0$ with Lie brackets as \eqref{super-reduced-Lie_brackets}, and $\omega$ is given by \eqref{Tisometrico}. 
%	Let $T$ be associated CKY tensor, that is, $\omega(x,y)=\la Tx,y\ra$ $x,y\in \ggo$. 
	Then, by a direct computation, we have that
	\begin{align*}
	\mathrm{d}\omega(x,z_2,\xi)=\la [\xi,x],Tz_2\ra=-\frac{{6}r}{s},
	\end{align*}
%    By a direct computation we can show that $$\mathrm{d}\omega=\frac{{6}r}{s}(),$$
 and therefore $\omega$ is never closed.
\end{proof}

Now, in order to determine the vector space $\mathcal{CKY}^2(\ggo_{r,s},\ip)$, we investigate KY $2$-forms on the metric Lie algebras $(\ggo_{r,s},\ip)$, or equivalently KY tensors on  $(\ggo_{r,s},\ip)$. Note that the KY condition on a metric Lie algebra is  given  by  the CKY condition \eqref{ckyequationliealgebra} with $\theta=0$.

\begin{prop}\label{kyT0}
	Let  $(\ggo_{r,s},\ip)$ be the  metric Lie algebra with Lie brackets given by \eqref{super-reduced-Lie_brackets}.
	If $T$ is a KY tensor on $(\ggo_{r,s},\ip)$, then $T=0$. In particular, $(\ggo_{r,s},\ip)$ does not admit non-zero parallel tensors.
\end{prop}
%O lo enunciamos asi, o sino hay que decir: una $(\ggo,\ip)$ que admite un strcit CKY, prefiero como esta ahora

\begin{proof}
	Let $T$ be a KY tensor on $(\ggo_{r,s},\ip)$ with associated $2$-form 
	given by $$\omega=\theta\wedge(az^1+bz^2+cx^*+dy^*)+ z^1\wedge(ez^2+fx^*+gy^*)+z^2\wedge(hx^*+iy^*)+ x^*\wedge y^*,$$
	where $\{\theta,z^1,z^2,x^*,y^*\}$ is the metric dual
	of the orthonormal basis $\{ \xi,z_1,z_2,x,y\}$.
%	\[ T=\left[ \begin{matrix} 0&-a&-b&-c&-d\\
%	a&0&-e&-f &-g \\
%	b&e &0 & -h&-i  \\
%	c& f&h & 0& -j\\
%	d& g& i& j& 0
%	\end{matrix}\right]\]
%$\omega=\theta\wedge (az^1+bz^2+cx^*+dy^*) + z^1\wedge(ez^2+fx^*+gy^*)+ z^2\wedge(hx^*+iy^*)+jx^*\wedge y^*$	
	Then $\omega$ must be co-closed, which is equivalent to $c=j=d=0$.
	%It can be seen that $T$ preserves the center of $\ggo$ and its orthogonal complement $\zz^\perp$.
	On the other hand, from  \eqref{lema21-1} (with $\theta=0$) we have that $\la [Tz_1,u], z_2\ra=\la [Tz_2,u], z_1\ra=0$ and $\la [Tz_i,u], z_i\ra=0$ for $i=1,2$ and any $u\in\zz^\perp$.
	Then, using these equations with $u$ $w\in \{\xi,x,y\}$, we get that $a=b=f=g=h=i=0$.
	Finally, applying \eqref{ckyequationliealgebra} (with $\theta=0$) for $\{\xi,z_2,x\}$ we obtain $e=0$, and therefore $T=0$.
\end{proof}

\begin{rem} 
	An interesting question would be to try to find another metric on the Lie algebras $\ggo_{r,s}$ such that it admits a non-zero KY tensor. Of course, this metric Lie algebra will not admit any strict CKY tensor.
\end{rem}

As a direct consequence of Proposition \ref{kyT0} and Proposition \ref{never_closed} we have:

\begin{thm}\label{CKYspace}
	Let  $(\ggo_{r,s},\ip)$ be a metric Lie algebra with Lie bracket as in  \eqref{super-reduced-Lie_brackets}. Then, $\mathcal{KY}^2(\ggo_{r,s}\ip)=\mathcal{*KY}^2(\ggo_{r,s},\ip)=0$ and the space $\mathcal{CKY}^2(\ggo_{r,s},\ip)$ has dimension one.
\end{thm}

\begin{rem}
	As a consequence of Theorem \ref{CKYspace}, any CKY $2$-forms on $(\ggo_{r,s},\ip)$	cannot be written as a linear combination of Killing forms and Hodge dual of Killing forms. Compare with \cite[Theorem 4.1]{dBM21} where CKY $2$-forms on $2$-step nilpotent Lie algebras are studied, and it is shown that the space of CKY $2$-forms coincides with  the space of KY $2$-forms or with  the space of $*$-KY $2$-forms.
\end{rem}

It is known that the  Hodge-star operator interchanges $*$-KY forms and KY forms (see \cite[Theorem 2]{Stepanov}). 
Therefore, the next result follows from Theorem \ref{CKYspace}.

\begin{cor}\label{3forms}
 Let  $(\ggo_{r,s},\ip)$ be a metric Lie algebra with Lie bracket by \eqref{super-reduced-Lie_brackets}.  
Then, \linebreak $\mathcal{KY}^3(\ggo_{r,s}\ip)=\mathcal{*KY}^3(\ggo_{r,s},\ip)=0$ and the space $\mathcal{CKY}^3(\ggo_{r,s},\ip)$ has dimension one.
\end{cor}

\begin{rem}
	In \cite{dBM19}, the authors give a characterization of KY $3$-forms on $2$-step nilpotent nilmanifolds. %creo que si una métrica es naturally reductive, entonces siempre hay una 3-forma KY. como la de ustedes no tiene, entonces no sería naturally reductive
\end{rem}

	\begin{rem}
		Each metric Lie algebra in Table \ref{TensoresBaseCanonica} is isometrically isomorphic to one of the metric Lie algebras $(\ggo_{r,s},\ip)$ described above. Then it follows from Theorem \ref{CKYspace} and Corollary \ref{3forms} that the vector spaces $\mathcal{CKY}^p(L_{5,9},\ip_r)$,  $\mathcal{CKY}^p(\RR^2 \times\mathfrak{su}(2),\ip_{r,s})$ and $\mathcal{CKY}^p(\RR^2 \times\mathfrak{sl}(2,\RR),\ip_{r,s})$ have dimension one, for $p=2$ and $3$. In the same way, the vector spaces $\mathcal{KY}^p(L_{5,9},\ip_r)$,  $\mathcal{KY}^p(\RR^2 \times\mathfrak{su}(2),\ip_{r,s})$, $\mathcal{KY}^p(\RR^2 \times\mathfrak{sl}(2,\RR),\ip_{r,s})$, $\mathcal{*KY}^p(L_{5,9},\ip_r)$,  $\mathcal{*KY}^p(\RR^2 \times\mathfrak{su}(2),\ip_{r,s})$ and $\mathcal{*KY}^p(\RR^2 \times\mathfrak{sl}(2,\RR),\ip_{r,s})$ are trivial for $p=2$ and $3$.
\end{rem}

\medskip

\begin{rem}
	It can be shown that the Lie algebra of the holonomy group of all examples from Sections $3$ and $5$ is generic, i.e. $\mathfrak{hol}(G,g) \cong \mathfrak{so}(5)$ and $\mathfrak{hol}(G_{r,s},g) \cong \mathfrak{so}(5)$, where $(G,g)$ is the simply connected Lie group  with  left invariant metric  associated to any metric Lie algebra in Table \ref{tabla1} and $(G_{r,s},g)$ is the simply connected Lie group  with  left invariant metric  associated to any metric Lie algebra  $(\ggo_{r,s}, \ip)$. Then $(G,g)$ and $(G_{r,s},g)$ are irreducible as Riemannian manifolds.
\end{rem}

\

\end{document}